\documentclass[a4paper,10pt,intlimits,sumlimits]{amsart}
\usepackage[utf8]{inputenc}

\usepackage{amsmath}
\usepackage{amsthm}
\usepackage{amstext}
\usepackage{amssymb}
\usepackage{amsfonts}
\usepackage{mathrsfs}
\usepackage{mathtools}
\usepackage{enumerate}

\usepackage{esint}

\usepackage{bbm} 
\usepackage{bm}

\usepackage{tikz}
\usepackage{pgfplots}
\usetikzlibrary{patterns,calc,decorations.pathreplacing}

\usepackage{float}
\usepackage{graphicx}
\usepackage{wrapfig}

\usepackage[pdftex,bookmarks,colorlinks,breaklinks]{hyperref}  
\definecolor{dullmagenta}{rgb}{0.4,0,0.4}   
\definecolor{darkblue}{rgb}{0,0,0.4}
\definecolor{darkgreen}{rgb}{0,0.4,0}
\hypersetup{linkcolor=darkblue,citecolor=blue,filecolor=dullmagenta,urlcolor=darkblue} 

 \newcommand{\mc}{\mathcal}
 \newcommand{\ms}[1]{\mathscr{#1}}
 \newcommand{\mb}[1]{\mathbf{#1}}
 \newcommand{\mf}[1]{\mathfrak{#1}}

 \newcommand{\map}[3]{#1 \colon #2 \rightarrow #3}

\newcommand{\R}{\mathbb{R}} 
\newcommand{\RR}{\R}  
\newcommand{\C}{\mathbb{C}}
\newcommand{\CC}{\mathbb{C}} 
\newcommand{\N}{\mathbb{N}}
\newcommand{\NN}{\N} 
\newcommand{\Z}{\mathbb{Z}}
\newcommand{\ZZ}{\Z} 

\newcommand{\XX}{\mathbb{X}}

\newcommand{\E}{\mathbb{E}} 

\newcommand{\EE}{\mathfrak{E}} 

\newcommand{\WW}{\mathbb{W}} 

\newcommand{\PP}{\mathbb{P}} 
\newcommand{\tPP}{3\PP} 
\newcommand{\pw}{\mathfrak{p}} 

\renewcommand{\AA}{\mathbb{A}} 
\newcommand{\BB}{\mathbb{B}} 

\newcommand{\MM}{\mathbb{M}} 

\newcommand{\DD}{\mathbb{D}} 
\newcommand{\TT}{\mathbb{T}} 

\newcommand{\Emb}{\mc{E}}
\newcommand{\RS}{\mathbb{S}} 

\newcommand{\DEF}{\mathfrak{d}} 
\newcommand{\vs}[1]{{#1}^{\uparrow}} 

\newcommand{\Bor}{\ms{B}} 
\newcommand{\Sch}{\ms{S}} 

\newcommand{\1}{\mathbbm{1}}

\newcommand{\loc}{\operatorname{loc}}

\newcommand{\LL}{\hbox{\raisebox{0.06em}-}\kern-0.45emL}
\newcommand{\sL}{\hbox{\raisebox{0.06em}-}\kern-0.45emL}

\newcommand{\dd}{\mathrm{d}}  

\newcommand{\FT}{\widehat}

\renewcommand{\phi}{\varphi}

\renewcommand{\bar}{\overline}
\renewcommand{\tilde}{\widetilde}

\newcommand{\sm}{\setminus}

\DeclareMathOperator{\pv}{p.v.}

\DeclareMathOperator{\spt}{spt}

\DeclareMathOperator{\ch}{ch}

\DeclareMathOperator{\BHT}{BHT}
\DeclareMathOperator{\BHF}{BHF}
\DeclareMathOperator{\Mod}{Mod}
\DeclareMathOperator{\Tr}{Tr}
\DeclareMathOperator{\Dil}{Dil}

\newtheorem{thm}{Theorem}
\newtheorem*{thm*}{Theorem}

\newtheorem{defn}[thm]{Definition}
\newtheorem*{defn*}{Definition}

\newtheorem{prop}[thm]{Proposition}
\newtheorem*{prop*}{Proposition}

\newtheorem{cor}[thm]{Corollary}
\newtheorem*{cor*}{Corollary}

\newtheorem{lem}[thm]{Lemma}
\newtheorem*{lem*}{Lemma}

\newtheorem{rmk}[thm]{Remark}
\newtheorem*{rmk*}{Remark}

\numberwithin{equation}{section}
\numberwithin{thm}{section}

\usepackage{todonotes}

\begin{document}
 		
\title[Banach-valued modulation invariant Carleson embeddings]{Banach-valued modulation invariant Carleson embeddings and outer-$L^p$ spaces: the Walsh case}
\date{\today}

\author[A. Amenta]{Alex Amenta}
\address{\noindent Mathematisches Institut \newline \indent Universit\"at Bonn, Bonn, Germany}
\email{amenta@math.uni-bonn.de}

\author[G. Uraltsev]{Gennady Uraltsev}
\address{\noindent Department of Mathematics \newline \indent Cornell University, Ithaca, NY, USA}
\email{guraltsev@math.cornell.edu}

\subjclass[2010]{Primary: 42B20; Secondary: 42B25; 47A56}
\keywords{Tritile operator, bilinear Hilbert transform, time-frequency analysis, Walsh group, UMD Banach spaces, outer Lebesgue spaces, interpolation spaces}


\begin{abstract}
  We prove modulation invariant embedding bounds from Bochner spaces $L^p(\WW;X)$ on the Walsh group to outer-$L^p$ spaces on the Walsh extended phase plane.
  The Banach space $X$ is assumed to be UMD and sufficiently close to a Hilbert space in an interpolative sense.
  Our embedding bounds imply $L^p$ bounds and sparse domination for the Banach-valued tritile operator, a discrete model of the Banach-valued bilinear Hilbert transform.
\end{abstract}


\maketitle

\section{Introduction}
\label{sec:intro}
The bilinear Hilbert transform (BHT) of two complex-valued Schwartz functions $f_0,f_1 \in \Sch(\RR;\CC)$ is given by 
\begin{equation*} 
  \BHT(f_0,f_1)(x) := \operatorname{p.v.} \int_\RR f_0(x-t)f_1(x+t) \, \frac{\dd t}{t}.
\end{equation*}
The $L^p$ bounds 
\begin{equation}
\label{eq:BHT-Lp-bounds}
\| \BHT(f_{0},f_{1}) \|_{L^{p}(\R)} \lesssim_{p_0,p_1} \| f_{0} \|_{L^{p_{0}}(\R)} \| f_{1} \|_{L^{p_{1}}(\R)} \qquad \forall f_0,f_1 \in \Sch(\RR;\CC),
\end{equation}
with $p_0,p_1 \in (1,\infty]$ and $p \in (2/3,\infty)$ such that $p^{-1}=p_{0}^{-1}+p_{1}^{-1}$, were first proven by Lacey and Thiele \cite{LT97,LT99}.
Their proof extended techniques developed by Carleson and Fefferman in their proofs of Carleson's theorem on the almost-everywhere convergence of Fourier series \cite{lC66,cF73}. These techniques are now referred to as `time-frequency' or `wave packet' analysis.
In order to streamline and modularise these techniques, Do and Thiele developed a theory of `outer-$L^p$' spaces, yielding proofs of $L^p$ bounds for the BHT in which the key difficulties are cleanly compartmentalised \cite{DT15}.

The outer-$L^p$ technique is not applied directly to the BHT, but rather to its associated trilinear form $\BHF$, given by dualising with a third function $f_2\in \Sch(\RR;\CC) $:
\begin{equation}\label{eqn:tri-form}
  \BHF(f_0,f_1,f_2) := \int_\RR \pv  \int_\RR f_0(x-t)f_1(x+t)f_2(x) \, \frac{\dd t}{t} \, \dd x.
\end{equation}
For $p_{0},p_{1},p\in[1,\infty]$, the estimate (\ref{eq:BHT-Lp-bounds}) is equivalent to the bound
\begin{equation}\label{eqn:tri-form-Lp-bound}
  |\BHF(f_0,f_1,f_2)| \lesssim_{p_0,p_1,p} \prod_{u=0}^{2} \| f_u \|_{L^{p_u}(\R)} \qquad \forall f_0, f_1, f_2 \in \Sch(\RR;\CC).
\end{equation}
The trilinear form $\BHF$ is a nontrivial linear combination of the Hölder form (which satisfies the desired $L^p$ bounds by Hölder's inequality) and another trilinear form:
\begin{equation}\label{eqn:trilinear-forms}
  \begin{aligned}
    & \int_\RR f_0(x)f_1(x)f_2(x) \, \dd x -  \frac{i}{\pi} \BHF(f_0,f_1,f_2) \\
    & =  \int_{\R^{3}_{+}} \Emb[f_{0}](x,\eta-t^{-1},t)\,\Emb[f_{1}](x,\eta+t^{-1},t)\,\Emb[f_{2}](x,-2\eta,t) \,  \dd x  \, \dd \eta \,  \dd t
    \\
    & =: \int_{\R^{3}_{+}} \Emb_{0}[f_{0}](x,\eta,t)\,\Emb_{1}[f_{1}](x,\eta,t)\,\Emb_{2}[f_{2}](x,\eta,t) \, \dd x \, \dd\eta  \, \dd t.
  \end{aligned}
\end{equation}
Here $\R^{3}_{+} = \RR \times \RR \times (0,\infty)$ is the extended phase plane, which parametrises the underlying translation, modulation, and dilation symmetries of $\BHF$.
The functions $\mc{E}[f_u]$ are representations of the functons $f_u \in \Sch(\RR;\CC)$ as functions $\R^{3}_{+} \to \CC$.
We refer to $\mc{E}$ as an \emph{embedding map}; the modified embedding maps $\mc{E}_u$ differ from $\mc{E}$ by a change of variables.\footnote{The precise definitions of $\mc{E}$ and $\mc{E}_u$ are not important for this introduction. \cite[(6.1)]{DT15} gives one possible definition.}

The outer-$L^p$ technique factorises $L^p$ bounds for the trilinear form  in \eqref{eqn:trilinear-forms} into a chain of inequalities:
\begin{align*}
    \bigg|\int_{\R^{3}_{+}} \Emb_0[f_{0}]\,\Emb_1[f_{1}]\,\Emb_2[f_{2}] \, \dd x \, \dd \eta \, \dd t\bigg|
    \lesssim \|  \Emb_0[f_{0}]\Emb_1[f_{1}]\Emb_2[f_{2}]\|_{L^{1}_{\nu}\sL^{1}_{\mu}S^{1}}
    \\
    \lesssim \prod_{u=0}^{2} \| \Emb_u[f_{u}] \|_{L^{p_{u}}_{\nu}\sL^{q_{u}}_{\mu}\RS_{u}} \lesssim\prod_{u=0}^{2} \| f_{u} \|_{L^{p_{u}}(\RR) }&.
\end{align*}
The first inequality is a Radon--Nikodym-style domination; the classical integral over $\R^{3}_{+}$ is controlled by an \emph{iterated outer-$L^1$} quasinorm $\|\cdot\|_{L^{1}_{\nu}\sL^{1}_{\mu}S^{1}}$.
This quasinorm is defined with respect to certain outer measures $\mu$ and $\nu$ on $\R^{3}_{+}$, along with a \emph{size} $S^1$ on functions on $\R^{3}_{+}$, which measures functions on distinguished subsets of $\R^{3}_{+}$.
The second inequality is a H\"older inequality for the iterated outer-$L^p$ quasinorms. This involves further sizes $\RS_{u}$, $u \in \{0,1,2\}$, which are connected to the size $S^1$ by a `size-H\"older' inequality. The first two inequalities follow from general properties of outer-$L^p$ spaces. The third inequality follows from the bounds
\begin{equation}\label{eqn:intro-emb}
  \|\Emb_u[f]\|_{L^{p_{u}}_{\nu}\sL^{q_{u}}_\mu \RS_{u}} \lesssim \|f\|_{L^{p_u}(\RR)} \qquad \forall f \in \Sch(\RR;\CC),
\end{equation}
which carry most of the difficulty of the problem.
These are \emph{modulation invariant Carleson embedding bounds}, so named because the operators $\mc{E}_u$ are modulation invariant in the sense that
\begin{equation*}
  \Emb_u[e^{2\pi i z \cdot \xi} f(z)](x,\eta,s) = \Emb_u[f](x,\eta+\xi,s),
\end{equation*}
and the outer-$L^p$ quasinorms $\|\cdot\|_{L_\nu^{p_u} \sL_\mu^{q_u} \RS_{u}}$ are invariant with respect to translation in the second variable.\footnote{These embeddings are also translation and dilation invariant, as are classical Carleson embeddings.} These bounds do not follow from general properties of outer-$L^p$ spaces, making them an interesting object of study in their own right. The abstract outer-$L^p$ theory offers one useful reduction in this direction: to prove the bounds \eqref{eqn:intro-emb}, it suffices to prove weak endpoint bounds and argue by an outer-$L^p$ version of the Marcinkiewicz interpolation theorem.

In this paper we consider functions $f \colon \RR \to X$ valued in a complex Banach space $X$.
Banach-valued analysis has a rich history which we do not attempt to summarise here; we simply point the reader to the recent volumes \cite{HNVW16,HNVW17}.
Embedding bounds from the Bochner space $L^p(\RR;X)$ into outer-$L^p$ spaces on the upper half-space $\RR \times (0,\infty)$ have been proven by Di Plinio and Ou \cite{DPO18}, with applications to Banach-valued multilinear singular integrals; the upper half-space parametrises translation and dilation symmetries, but not modulation symmetries.
We would like to prove such embedding bounds into outer-$L^p$ spaces on $\RR \times \RR \times (0,\infty)$, in order to incorporate modulation invariance.\footnote{Between the submission and acceptance of this paper, we successfully proved such embedding bounds \cite{AU19-2}, establishing $L^p$-bounds for UMD-valued bilinear Hilbert transforms. The bounds for bilinear Hilbert transforms, and more general multilinear Fourier multipliers, were simultaneously proven by Di Plinio, Li, Martikainen, and Vuorinen \cite{DPLMV19-3} using a different method.}
As a first step we prove these for a discrete model of the real line---the $3$-Walsh model---in which many technical difficulties in time-frequency analysis are removed, while the core features of the analysis remain.

In the $3$-Walsh model, the real line is replaced by the $3$-Walsh group
\begin{equation}\label{eq:3walsh}
  \WW = \WW_3 := \bigg\{ x \in \prod_{n \in \ZZ} \ZZ / 3\ZZ : \text{$[x]_n = 0$ for $n$ sufficiently large}\bigg\}, 
\end{equation}
where $[x]_n$ denotes the $n$-th component of $x$, and with group operation inherited from the infinite product. Up to measure zero, the Haar measure on $\WW$ can be identified with the Lebesgue measure on $[0,\infty)$, and the group operation on $\WW$ corresponds to ternary digitwise addition modulo $3$ (i.e. ternary digitwise addition without carry) on $[0,\infty)$. The dual group of $\WW$ can be identified with $\WW$ itself, and the Walsh--Fourier transform of the characteristic function of a triadic interval is again such a characteristic function (up to a suitable Walsh modulation). Thus one can construct `Walsh wave packets', supported in a given triadic interval of $[0,\infty)$, with frequency support in another given triadic interval. In this way, $\WW$ supports an idealised time-frequency analysis that is not possible on $\RR$, as a compactly supported function on $\RR$ cannot have compactly supported Fourier transform.

We work with the $3$-Walsh group, although the $2$-Walsh group (defined by replacing $3$ by $2$ in the definition) is more commonly used in the literature.
We have made this choice because the $3$-Walsh group leads to a more natural discrete model of $\BHF$ than the $2$-Walsh group.
Our arguments work equally well for any choice of integer parameter greater than or equal to $2$.

In our application of the $3$-Walsh model, the role of the extended phase plane $\R^{3}_{+}$ is taken by the set $\tPP$ of all \emph{tritiles}: i.e. the set of all rectangles $\mb{P} = I_\mb{P} \times \omega_\mb{P}$ of area $3$ in $[0,\infty) \times [0,\infty)$ (identified with $\WW \times \WW$) whose sides are triadic intervals. The tritile $\mb{P}$ roughly corresponds to the point $(x_{\mb{P}},\xi_{\mb{P}},|I_{\mb{P}}|)$, where $x_{\mb{P}}$ and $\xi_{\mb{P}}$ are the centres of $I_{\mb{P}}$ and $\omega_{\mb{P}}$ respectively. Each tritile is split into three \emph{tiles} $\mb{P}_u = I_{\mb{P}} \times \omega_{\mb{P}_u}$, $u \in \{0,1,2\}$---i.e. rectangles of area $1$ with triadic sides---all with the same time interval $I_{\mb{P}}$, and to each of these tiles is associated a Walsh wave packet $\map{w_{\mb{P}_u}}{\WW}{\CC}$, supported in $I_{\mb{P}}$ with frequency support in $\omega_{\mb{P}_v}$. The embedding $\mc{E}[f] \colon \tPP \to X^{3}$ of a function $f \colon \WW \to X$ is given by integrating $f$ against the three wave packets corresponding to a given tritile $\mb{P}$, and collecting the results in a triple
\begin{equation}\label{eq:embedding}
  \Emb[f](\mb{P}) := \big(   \langle f; w_{\mb{P}_u} \rangle \big)_{u\in\{0,1,2\}}.
\end{equation}
There are two equivalent ways of looking at this embedding: either as an $X^3$-valued function on tritiles, or as an $X$-valued function on tiles, where we write
\begin{equation*}
  \Emb[f](P) = \langle f; w_P \rangle = \Emb[f](\mb{P})_u,
\end{equation*}
where $\mb{P}$ is the unique tritile containing the tile $P$, and $u$ is the index such that $P = \mb{P}_u$.
Both viewpoints are handy, and we switch between them freely.

The main results of this paper are the following embedding bounds.
The Banach space assumptions (UMD, $r$-Hilbertian) are explained in Section \ref{sec:vvwpa}, and the relevant outer structures on $\tPP$ in Section \ref{sec:outermeasures}.
The sizes $\RS$ are also defined in Section \ref{sec:outermeasures}; they depend on the Banach space $X$ appearing in the statement of the theorem, although this is not apparent from the notation.

\begin{thm}\label{thm:intro-embeddings}
  Let $X$ be a Banach space which is UMD and $r$-Hilbertian for some $r \in [2,\infty)$.  Then for all convex sets $\AA \subset \tPP$ of tritiles, the following embedding bounds hold.
  \begin{itemize}
  \item For all $p \in (r,\infty)$,
    \begin{equation}\label{eqn:intro-emb-1} 
      \bigl\|\1_{\AA}\,\Emb[f] \bigr\|_{L_\mu^p \RS} \lesssim \|f \|_{L^p(\WW;X)} \qquad \forall f \in \Sch(\WW;X).
    \end{equation}
  \item For all $p \in (1,\infty)$ and all $q \in (\min(p,r)^\prime(r-1),\infty)$,
    \begin{equation}\label{eqn:intro-emb-2}
      \bigl\|\1_{\AA}\,\Emb[f]\bigr\|_{L_\nu^{p} \sL_\mu^q \RS} \lesssim \|f\|_{L^p(\WW;X)} \qquad \forall f \in \Sch(\WW;X).
    \end{equation}
  \end{itemize}
  The implicit constants in the above bounds do not depend on $\AA$.
\end{thm}

The set of exponents for which the embedding bounds \eqref{eqn:intro-emb-2} hold is sketched in Figure \ref{fig:intro-emb-exponents}; in the dotted region, the iterated embedding bounds basically correspond to the non-iterated bounds.
For $p \leq r$ we only have embeddings into iterated outer-$L^p$ spaces; such behaviour `outside local $L^r$' necessitated the introduction of iterated outer-$L^p$ spaces by the second author \cite{gU16}.


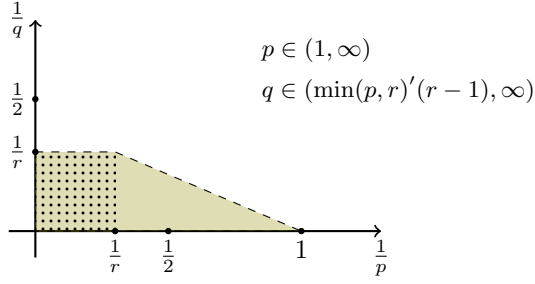
\begin{figure}
  \begin{tikzpicture}[scale=3.5]

    \def\xmax{1.3} \def\xmin{-0.1} \def\ymax{0.8} \def\ymin{-0.1}
    
    \def\r{0.3}  

    \draw [->, thick] (0, \ymin) -- coordinate (y axis mid) (0,\ymax);
    \draw [->, thick] (\xmin, 0) -- coordinate (x axis mid) (\xmax,0);
    \node [right] at (0.8,0.6) {\small $\begin{aligned} &p\in(1 , \infty)
        \\ & q\in(\min(p,r)^\prime(r-1),\infty)\end{aligned}$};
    
    \fill[color=olive,opacity=0.3] (\r,0) -- (\r,\r) --
    (1,0) -- cycle;

    \fill[color=olive,opacity=0.3] (0,0) -- (0,\r) -- (\r,\r) -- (\r,0) -- cycle;
    \fill[pattern=dots, pattern color=black] (0,0) -- (0,\r) -- (\r,\r) -- (\r,0) -- cycle;

    \draw[fill=black] (1,0) circle (0.01) node [below] {$1$};

    \draw[fill=black] (0.5,0) circle (0.01) node [below] {$\frac{1}{2}$};
    \draw[fill=black] (0,0.5) circle (0.01) node [left] {$\frac{1}{2}$};
    
    \draw[fill=black] (\r,0) circle (0.01) node [below] {$\frac{1}{r}$}; 
    \draw[fill=black] (0,\r) circle (0.01) node [left] {$\frac{1}{r}$};

    \node[below] at (1.3,0) {$\frac{1}{p}$};
    \node[left] at (0,\ymax) {$\frac{1}{q}$};




    \draw[black,thin,dashed] (0,\r) -- (\r,\r) -- (1,0) -- (0,0) -- cycle;


  \end{tikzpicture}
  \caption{Exponents $(p,q)$ for which \eqref{eqn:intro-emb-2} holds.} 
  \label{fig:intro-emb-exponents}
\end{figure}

Let us now discuss Banach-valued versions of the trilinear form $\BHF$.
Consider a triple of Banach spaces $(X_0,X_1,X_2)$ and a bounded trilinear form
\begin{equation}\label{eqn:intro-tri-form}
  \Pi \colon X_0 \times X_1 \times X_2 \to \CC.
\end{equation}
With respect to this data we define
\begin{equation*}
  \BHF_\Pi(f_0,f_1,f_2) :=\pv \int_\RR \int_\RR \Pi(f_0(x-t),f_1(x+t),f_2(x)) \, \frac{\dd t}{t} \, \dd x
\end{equation*}
for $f_u \in \Sch(\RR;X_u)$, $u \in \{0,1,2\}$. The first $L^p$-bounds for $\BHF_\Pi$ were proven by Silva, in the case $X_0 = \ell^R$, $X_1 = \ell^\infty$, $X_2 = \ell^{R^\prime}$, for $R \in (4/3,4)$, with $\Pi$ the natural product-sum map \cite[Theorem 1.7]{pS14}.
The set of allowed Banach spaces was extended by Benea and Muscalu using a new `helicoidal method' \cite{BM16,BM17}, and by Lorist and Nieraeth by Rubio de Francia-type extrapolation methods \cite{LN19,bN19} relying on weighted estimates for $\BHF$ as proven for example in \cite{CUM18,CDO18}.
One limitation of these results is that they only hold when the spaces $X_0,X_1,X_2$ are Banach lattices, excluding interesting examples such as the Schatten classes $\mc{C}^p$ and more general non-commutative $L^p$ spaces.
It remains an open question as to whether there are any $L^p$-bounds for $\BHF_\Pi$ without this limitation.

As a corollary of Theorem \ref{thm:intro-embeddings}, we prove $L^p$-bounds for the $3$-Walsh model of $\BHF_\Pi$ without assuming any lattice structure. This model is the \emph{tritile form} $\Lambda_\Pi$, defined by
\begin{equation}\label{eqn:tritile-form}
  \begin{aligned}
    \Lambda_\Pi(f_0,f_1,f_2) &:= \sum_{\mb{P} \in \tPP} \Pi\Big(\langle f_0; w_{\mb{P}_0} \rangle, \langle f_1; w_{\mb{P}_1} \rangle, \langle f_2; w_{\mb{P}_2} \rangle \Big) |I_{\mb{P}}| \\
    &= \sum_{\mb{P} \in \tPP} \Pi\Big(\Emb[f_0](\mb{P}_0), \Emb[f_1](\mb{P}_1), \Emb[f_2](\mb{P}_2) \Big) |I_{\mb{P}}|
\end{aligned}
\end{equation}
for $f_u \in \Sch(\WW;X_u)$, $u \in \{0,1,2\}$.\footnote{It is not obvious that the sum on the right hand side converges absolutely; see \cite[Lemma 5.1]{HLP13} for a proof of this convergence for the quartile form, which will be discussed later in the introduction. Of course, the absolute convergence follows from our theorem.}

\begin{thm}\label{thm:intro-tritile-boundedness}
    Let $(X_u)_{u\in\{0,1,2\}}$ be UMD Banach spaces, such that each $X_u$ is $r_u$-Hilbertian for some $r_u \in [2,\infty)$, and let $\Pi \colon X_0 \times X_1 \times X_2 \to \CC$ be a bounded trilinear form. Given any  H\"older triple of exponents  $(p_u)_{u \in \{0,1,2\}}\in(1,\infty)^{3}$  satisfying
  \begin{equation}\label{eqn:condn}
    \sum_{u=0}^{2} \frac{1}{\min(p_u,r_u)^\prime(r_u-1)} > 1,
  \end{equation}
  we have the bound
  \begin{equation*}
    |\Lambda_\Pi(f_0,f_1,f_2)| \lesssim \prod_{u=0}^{2} \|f_u\|_{L^{p_u}(\WW;X_u)} \qquad \forall f_u \in \Sch(\WW;X_u).
  \end{equation*}
\end{thm}

The region of exponents $(p_u)_{u=0}^2$ for which this theorem holds (more precisely, the region of their reciprocals) is characterised as the interior of a polygon in Section \ref{sec:Lpbounds}.
This region is only nonempty when the Hilbertian exponents $(r_u)_{u=0}^2$ are jointly sufficiently close to $2$, in the sense that
\begin{equation*}
  \sum_{u=0}^2 \frac{1}{r_u} > 1.
\end{equation*}

$L^p$ bounds for the Banach-valued quartile form (the $2$-Walsh analogue of $\Lambda_\Pi$) were first established by Hytönen, Lacey, and Parissis \cite{HLP13}. Their assumptions on the Banach spaces $X_u$ are very similar to ours---possibly equivalent, although this is not known---and the resulting range of exponents in their $L^p$ bounds are the same as ours when restricted to the reflexive range (see Section \ref{sec:Lpbounds}).\footnote{It was pointed out to us by the anonymous referee that the sparse domination obtained in Theorem \ref{thm:intro-sparse-bounds} implies estimates outside the reflexive range, and fully recovers the range of estimates obtained in \cite{HLP13}.}
Banach-valued time-frequency analysis was initiated by Hytönen and Lacey in their work on the Carleson operator, and continued with their work with Parissis on the Walsh model of the variational Carleson operator \cite{HL13, HL18, HLP14}.
We have taken substantial inspiration from these papers.

The iterated embeddings of Theorem \ref{thm:intro-embeddings} imply not only $L^p$ bounds for the tritile form, but also sparse domination.
The connection between sparse domination and Carleson embeddings into iterated outer-$L^p$ spaces was first shown by Di Plinio, Do, and the second author \cite{DPDU18}.
A collection of intervals $\mc{G}$ in $\WW$ is \emph{sparse} if 
\begin{equation*}
  \| \mc{G} \|_{sp}:= \sup_{I\subset \WW} \frac{1}{|I|}\sum_{\substack{J \in \mc{G} \\ J \subset I}} |J| < \infty
\end{equation*}
where the supremum is taken over all intervals $I \subset \WW$ (see \cite[\textsection 6]{LN18} or \cite{TH18} for a proof that this is equivalent to the more familiar definition of a sparse collection).

\begin{thm}\label{thm:intro-sparse-bounds}
  Let $(X_u)_{u \in \{0,1,2\}}$, $(r_u)_{u \in \{0,1,2\}}$, and $\Pi$ be as in Theorem \ref{thm:intro-tritile-boundedness}. Let $(p_{u})_{u\in{0,1,2}}$ be any triple of exponents satisfying \eqref{eqn:condn}.  Then
  \begin{equation*}
    |\Lambda_\Pi(f_0,f_1,f_2)| \lesssim \sup_{\| \mc{G} \|_{sp}\leq 1} \sum_{I \in \mc{G}} |I| \prod_{u=0}^2 \bigg( \fint_I \|f_u\|_{X_u}^{p_u} \bigg)^{1/p_u} \qquad \forall f_u \in \Sch(\WW;X_u)
  \end{equation*}
  where the supremum is taken over all sparse collections of intervals $\mc{G}$.
\end{thm}

The term appearing on the right of the bound of Theorem \ref{thm:intro-sparse-bounds} is referred to as a sparse form. It is straightforward to show that sparse forms satisfy the bounds
\begin{equation*}
  \sum_{I\in\mc{G}} |I| \prod_{u=0}^{2} \Bigl( \fint_{I}|f_{u}|^{p_{u}} \Bigr)^{1/p_{u}} \lesssim \| \mc{G} \|_{sp} \prod_{u=0}^{2}\| f_{u} \|_{L^{\bar{p}_{u}}} 
\end{equation*}
for any Hölder triple of exponents $(\bar{p}_{u})_{u\in\{0,1,2\}}$ with $\bar{p}_{u}>p_{u}$.
Furthermore, sparse forms satisfy various weighted bounds, which we do not pursue here; for more information see for example \cite{mL19, LN18, bN19}.

Let us return to the assumptions of Theorem \ref{thm:intro-tritile-boundedness}: we consider three UMD Banach spaces $(X_u)_{u = 0,1,2}$, each of which is $r_u$-Hilbertian, linked with a bounded trilinear form $\map{\Pi}{X_0 \times X_1 \times X_2}{\CC}$.
There are a few natural examples that one should keep in mind:
\begin{itemize}
\item Let $X$ be a UMD Banach space which is $r$-Hilbertian.
  Then the dual space $X^*$ is also UMD and $r$-Hilbertian, and we can consider the `duality trilinear form'
  \begin{equation*}
    \map{\Pi}{X \times X^* \times \CC}{\CC}, \qquad \Pi(x,x^*,\lambda) := \lambda x^*(x).
  \end{equation*}
  Since $\CC$ is UMD and $2$-Hilbertian (i.e. Hilbert), the corresponding region of exponents in Theorem \ref{thm:intro-tritile-boundedness} is nonempty provided
  \begin{equation*}
    \frac{2}{r} + \frac{1}{2} > 1,
  \end{equation*}
  i.e. when $r < 4$.

\item Consider a H\"older triple of exponents $r_0, r_1, r_2 \in (1,\infty)$, so that the Lebesgue spaces $L^{r_u}(\R)$ are UMD and $\max(r_u,r_u^\prime)$-Hilbertian, and there is no exponent $r < \max(r_u,r_u^\prime)$ such that $L^{r_u}(\R)$ is $r$-Hilbertian.
  Consider the `integration trilinear form'
  \begin{equation*}
    \map{\Pi}{L^{r_0}(\R) \times L^{r_1}(\R) \times L^{r_2}(\R)}{\CC}, \qquad \Pi(f,g,h) := \int_\Sigma f(x)g(x)h(x) \, \dd x.
  \end{equation*}
  Then Theorem \ref{thm:intro-tritile-boundedness} would yield a nontrivial region of exponents provided that
  \begin{equation*}
    \sum_u \max(r_u,r_u^\prime)^{-1} > 1.
  \end{equation*}
  But since $\sum_u r_u = 1$, this occurs only if $\max(r_u,r_u^\prime) < r_u$ for some $u$, which is impossible.
  Thus this trilinear form never fits into our framework.
  This is in stark contrast with the results of Benea and Muscalu, who obtain bounds for $\BHF_\Pi$ for this trilinear form for any H\"older triple $r_0,r_1,r_2$ with $r_0,r_1 \in (1,\infty]$ and $r_2 \in [1,\infty)$ \cite{BM16}.
  The reason for this discrepancy is our reliance on UMD methods.

\item On the other hand, replacing $\R$ with $\NN$ in the preceding example, one can define the integration trilinear form on $\ell^{r_0} \times \ell^{r_1} \times \ell^{r_2}$ provided that $\sum r_u^{-1} \geq 1$.
  Thus this trilinear form fits into our framework provided that $\sum r_u^{-1} > 1$ and $r_u \geq 2$ for each $u$.
  The same holds when each $\ell^{r_u}$ is replaced by the Schatten class $\mc{C}^{r_u}$ and $\Pi$ is replaced by the `composition trilinear form'.
\end{itemize}

Here is a brief overview of the paper.
In Section \ref{sec:prelims} we introduce the basics of the Walsh group $\WW$ and the associated time-frequency analysis.
In Section \ref{sec:vvwpa} we discuss various Banach space properties and their analytic consequences.
In Section \ref{sec:outermeasures} we set up the framework of outer structures and outer-$L^p$ spaces.
Of particular importance are the size-Hölder inequality (Proposition \ref{prop:new-size-holder}) for the `randomised' sizes, and the size domination theorem (Theorem \ref{thm:size-domination}), which lets us control the randomised sizes by a simpler `deterministic' size.
Section \ref{sec:energy-embedding} is devoted to proving Theorem \ref{thm:intro-embeddings}.
Crucial to these arguments is a basic tile selection algorithm given in Proposition \ref{prop:tile-selection}.
This is a simpler version of a more familiar `tree selection algorithm' often used in time-frequency analysis; the simplification is thanks to the aforementioned size domination theorem.
Finally, in Section \ref{sec:tritile-form}, we deduce $L^p$ bounds and sparse domination for the tritile form.
Section \ref{sec:RMF} is an appendix, in which we sketch an alternative method using $R$-bounds and the RMF property; this requires additional Banach space assumptions, but the proof is a bit more direct.

\subsection{Notation}\label{sec:notation}

The letter $\WW$ will always stand for the $3$-Walsh group $\WW_3$; we always write $\WW_{\mf{p}}$ when we want to use a different parameter $\mf{p}$ (see Section \ref{sec:prelims}).
For a Banach space $X$ and $p \in [1,\infty]$, $L^p(\WW;X)$ denotes the Bochner space of strongly measurable functions $\WW \to X$ such that the function $x \mapsto \|f(x)\|_X$ is in the usual Lebesgue space $L^p(\WW)$.
For technical details on Bochner spaces see \cite[Chapter 1]{HNVW16}.
When $I \subset \WW$ is an interval and $f \in L_{\loc}^p(\WW;X)$, we let
\begin{equation}\label{eq:p-avg}
  \|f\|_{\sL^p(I;X)} := \bigg( \frac{1}{|I|} \int_I \|f(x)\|_X^p \, \dd x \bigg)^{1/p} 
\end{equation}
denote the $L^p$-average, and
\begin{equation}\label{eq:p-max-fun}
  M_p f(x) := \sup_{I \ni x} \|f\|_{\sL^p(I;X)} \qquad  \forall x \in \WW
\end{equation}
denote the triadic $p$-maximal function; the supremum is taken over all intervals $I \in \WW$ containing $x$.
For $f \in L_{\loc}^1(\WW;X)$ we let
\begin{equation*}
  \langle f \rangle_I := \frac{1}{|I|} \int_I f(x) \, \dd x \in X
\end{equation*}
denote the average, of $f$ on $I$.
For $f\in \Sch(\WW;X)$ and $g\in\Sch(\WW;\C)$ let
\begin{equation*}
  \langle f; g \rangle := \int_\WW f(x)g(x) \, \dd x \in X.
\end{equation*}

We say that a triple of exponents $(p_u)_{u \in\{0,1,2\}}$ with $p_u \in [1,\infty]$ is a \emph{Hölder triple} if $\sum_{u=0}^2 p_u^{-1} = 1$.

Throughout the paper, we use $(\varepsilon_n)_{n \in A}$ to denote a sequence of independent Rademacher variables (i.e. random variables that take the values $\pm1$ with equal probability), indexed over some countable indexing set $A$. It never matters precisely which probability space these Rademacher variables live on. We denote the expectation over this probability space by $\E$.

\subsection{Acknowledgements}

Part of this research was completed while the first author was a postdoctoral researcher at the TU Delft and while the second author was a doctoral student at Bonn International Graduate School.

The first author was supported by the VIDI subsidy 639.032.427 of the Netherlands Organisation for Scientific Research (NWO) and a Fellowship for Postdoctoral Researchers from the Alexander von Humboldt Foundation.
We thank Mark Veraar and Christoph Thiele for their encouragement and suggestions, 



 \section{Walsh time-frequency analysis}
 \label{sec:prelims}
 In this section we introduce the Walsh group $\WW$ and the extended Walsh phase plane.
In particular we introduce tiles, wave packets, tritiles, trees, and strips; none of this material is new; we include it here for the convenience of the reader, and to fix notation.
In Subsection \ref{sec:defect} we introduce the defect operator, which is an important technical tool in our analysis.

\subsection{The Walsh group}

Fix an integer $\pw\geq 2$.
The Walsh group $\WW_{\pw}$ is
\begin{equation}
  \label{eq:p-walsh}
  \begin{split}
    \WW_{\pw} := \Bigl\{x\in  \prod_{n\in\Z} \Z/\pw\Z  \colon |x|<\infty\Bigr\}, 
  \end{split}
\end{equation}
where $|x|=\max \bigl\{\pw^{n} \colon [x]_{n}\neq0\bigr\}$ and $[x]_{n}$ is the $n$\textsuperscript{th} component (`digit') of $x$.
The group operation $+$ is the digit-wise addition in $\Z/\pw\Z$, and the map $(x,y)\mapsto |x-y|$ is a translation invariant metric on $\WW_{\pw}$, giving $\WW_{\pw}$ the structure of a locally compact abelian group, and thus guaranteeing the existence of a Haar measure on $\WW_{\pw}$.
We normalise this measure so that $|B_{1}(0)|=1$, and it follows that
\begin{equation}
  \label{eq:haar-measure}
   |B_{\pw^{n}}(x)|=\pw^{n} \qquad \forall x \in \WW_{\pw}.
 \end{equation}

 As explained in the introduction, there is a correspondence between the Walsh group and the non-negative reals $[0,\infty)$ given by the surjective map
 \begin{equation*}
   \WW_{\pw}\ni x \mapsto  \sum_{n \in \ZZ} [x]_n \pw^n \in [0,\infty), 
 \end{equation*}
 which is injective up to a set of measure zero.
 The pullback of the Lebesgue measure by this map is the Haar measure on $\WW_{\pw}$, and intervals in $[0,\infty)$ correspond to balls in $\WW_{\pw}$.
 Thus we often refer to Walsh balls as intervals.
 
 Let $X$ be a Banach space.
 We say that a function $\map{f}{\WW_{\pw}}{X}$ is \emph{Schwartz}, denoted $f \in \Sch(\WW_{\pw};X)$, if there exists $N>0$ such that $f$ is supported on $B_{\pw^{N}}(0)$ and constant on any interval $I$ with $|I|<\pw^{-N}$. For all $p \in [1,\infty)$, the Schwartz functions are dense in $L^p(\WW_{\pw};X)$. 

 The dual group of $\WW_{\pw}$ can be identified with $\WW_{\pw}$ itself, and the characters of $\WW_{\pw}$ are the Walsh exponentials
\begin{equation}\label{eqn:walsh-exp}
  \exp_\xi(x) := (e^{2\pi i / \pw})^{\sum_{j+k = -1}[\xi]_j [x]_k}\qquad \forall x,\xi \in \WW_{\pw}. 
\end{equation}
The Walsh--Fourier transform of a function $f \in L^1(\WW_{\pw};\CC)$ is thus 
\begin{equation}\label{eq:WFT}
  \FT f(\xi) := \int_{\WW_{\pw}} f(x) \bar{\exp_{\xi}(x)} \, \dd x \qquad \forall  \xi \in \WW_{\pw},
\end{equation}
and  we have the Plancherel identity
\begin{equation*}
  \int_{\WW_{\pw}} \FT f(\xi) \bar{\FT g(\xi)} d \xi = \int_{\WW_{\pw}} f(x) \bar{g(x)} \, \dd x \qquad \forall f,g \in \Sch(\WW_{\pw};\CC).
\end{equation*}
Consider the \emph{modulation}, \emph{translation}, and \emph{dilation} operators on functions $f \colon \WW_{\pw} \to \CC$, given by
\begin{equation}
  \label{eq:symmetries}
  \begin{aligned}
    & \Mod_{\eta}f(x) := \exp_\eta(x) f(x) & \qquad & \forall \eta \in \WW_{\pw},\\
    & \Tr_{y}f(x) := f(x-y) & \qquad & \forall y \in \WW_{\pw},\\
    & \Dil_{\pw^{n}} f (x):= \pw^{-n} f(\pw^{-n} x) & \qquad & \forall n\in\Z,
  \end{aligned}
\end{equation}
where  $[\pw^{-n}x]_{j}:=[x]_{j+n}$ for all $n,j \in \Z$. It follows from the definition of the Walsh--Fourier transform that
\begin{equation}
  \label{eq:fourier-symmetries}
  \begin{aligned}
    & \FT {\Mod_{\eta}f} = \Tr_{y}\FT{f} \\
    & \FT{\Tr_{y}f} = \Mod_{y}\FT{f} \\
    & \FT{\Dil_{\pw^{n}}f} = \pw^{-n} \Dil_{\pw^{-n}} \FT{f}
  \end{aligned}
  \qquad \forall f \in \Sch(\WW_{\pw};\CC).
\end{equation}

Given two intervals $I, I'\subset \WW$, it holds that
\begin{equation*}
  I\cap I'\neq \emptyset \quad \implies\quad I\subseteq I'\text { or } I'\subseteq I.
\end{equation*}
This is a familiar property of $\pw$-adic intervals in $[0,\infty)$.  Each interval has $\mf{p}$ child intervals $\{\ch_{0}(I),\ch_{1}(I),\ldots,\ch_{\mf{p}-1}(I)\}=\ch(I)$ given by
\begin{equation}
  \label{eq:children}
  \ch_{j}(I)=\{x\in I \colon [x]_{(\log_{\mf{p}}|I|) - 1}=j \}.
\end{equation}

\textbf{In the remainder of the paper we will work with the case $\pw = 3$, and we will write $\WW := \WW_3$.}

\subsection{The extended Walsh phase plane}

Strictly speaking, the extended Walsh phase plane is $\{(x,\xi,3^{n}) \in \WW \times \WW\times \R^{+} \colon n\in\Z\}$ where each point $(x,\xi,3^{n}) \in \WW \times \WW \times \R^{+}$ represents the time $x$, the frequency $\xi$, and the scale $3^n$. We can identify each point $(x,\xi,3^{n})$ with the rectangle $B_{3^{n}}(x)\times B_{3^{-n}}(\xi)\subset\WW \times \WW$; this provides for a more graphically intuitive way of thinking of time-frequency localisation. This identification is not injective, but it turns out that that this failure of injectivity correctly encodes the `uncertainty principle' i.e. the impossibility of determining both position (in time) and frequency to an arbitrary scale.

We thus introduce the notion of a \emph{tile}.

\begin{defn}[Tiles]\label{def:tile}
  A \emph{tile} is a rectangle $P = I_P \times \omega_P$ in $\WW \times \WW$ of area $1$, such that the sides $I_P$ and $\omega_P$ are intervals.
  We call $I_P$ the \emph{time interval} and $\omega_P$ the \emph{frequency interval} of $P$.
  For each tile $P$ there exist unique $x_{P},\xi_{P} \in \WW$ and $n\in\Z$ such that
  \begin{equation}
    \label{eq:tile-parameters}
    \begin{aligned}
      &  I_{P}=B_{|I_P|}(x_{P}), & & [x_{P}]_{j}=0 \text{ for } j < n, \\
      &  \omega_{P}=B_{|I_P|^{-1}}(\xi_{P}), & & [\xi_{P}]_{j}=0 \text{ for }  j < -n .
    \end{aligned}
  \end{equation}
  We call $x_P$ the \emph{centre} of the tile, and $\xi_{P}$ the \emph{frequency} of the tile.
  We denote the set of all tiles by $\PP$.
\end{defn}

To each tile $P$ we associate a wave packet $w_P$, which is a $\CC$-valued function supported in $I_P$ with frequency support $\omega_P$.
In the time-frequency sense, the wave packet $w_P$ is localised to $P$.

\begin{defn}[Wave packets]\label{def:wave-packet}
  Given a tile $P\in \PP$, the \emph{wave packet associated with $P$} is the function
  \begin{equation}
    \label{eq:wave-packet}
    w_{P}(x)= \Mod_{\xi_{P}}\Tr_{x_{P}}\Dil_{|I_{P}|}\1_{B_{1}(0)}(x)= |I_{P}|^{-1}\exp_{\xi_P}(x)\1_{I_{P}}(x).
  \end{equation}
  This is the unique function, up to multiplication by a unimodular constant, such that
  \begin{equation}
      \label{eq:wave-packet-spt}
      \spt w_{P} =I_{P}, \qquad  \FT{w_{P}} = \omega_{P}, \quad \text{and} \quad \| w_{P} \|_{L^1(\WW)}=1.
  \end{equation}
\end{defn}

\begin{rmk}\label{rmk:wave-packets-tiles}
It is convenient to identify wave packets with tiles, and thus to consider the translation, dilation, and modulation operators \eqref{eq:symmetries} as acting directly on tiles, so that for example
\begin{equation*}
  \Mod_{\xi} P = P' \iff \Mod_{\xi} w_{P} = c \, w_{P'} \text{ for some $|c|=1$.}
\end{equation*}
We could equivalently define our wave packets with an arbitrary choice of unimodular constant out the front; all the statements we make about wave packets will be invariant under this transformation.
In essence, what is most important is not the wave packet itself, but the subspace of $L^2(\WW;\CC)$ that it spans.
\end{rmk}

Simple support (and Walsh--Fourier support) considerations show that two tiles are disjoint if and only if their associated wave packets are orthogonal.
More refined statements can be made about the connection between tiles and wave packets. For example, a union of disjoint tiles $\bigcup_i P_i$ corresponds to the subspace of $L^2(\WW;\CC)$ spanned by the pairwise orthogonal wave packets $(w_{P_i})_i$, and this subspace does not depend on the specific representation of $\bigcup_i P_i$ as a disjoint union of tiles. In particular, if a tile $P$ is contained in such a union, then the wave packet $w_P$ can be written as a linear combination of the wave packets $w_{P_i}$. This is made precise in the following lemma.

\begin{lem}[Basis expansion of wave packets]\label{lem:wave-packet-interaction}
  Let $(P_{i})_{i\in\{1,\dots,N\}}$ be a finite collection of pairwise disjoint tiles.
  Then for any $P\subset \cup_{i=1}^{N} P_{i} $ it holds that
  \begin{equation}
    \label{eq:wave-packet-tile-proj}
    w_{P} = \sum_{i=1}^{N} \langle w_{P}; w_{P_{i}} \rangle w_{P_{i}} |I_{P_{i}}|.
  \end{equation}
\end{lem}

\begin{proof}
  We may assume that $P\cap P_{i}\neq \emptyset$ for all $i\in\{1,\dots,N\}$, for otherwise we would have $\langle w_P; w_{P_i} \rangle = 0$ and $P_i$ would not contribute to the right hand side of (\ref{eq:wave-packet-tile-proj}).

  If $I_{P_{i}}\subset I_{P} \subset I_{P_{j}}$ for $i\neq j$ then $P_{i}\cap P_{j}\neq \emptyset$, contradicting  the assumption, so  either $I_{P}\supset I_{P_{i}}$ for all $i\in\{1,\ldots,N\}$ or $I_{P}\subset I_{P_{i}}$ for all $i\in\{1,\ldots,N\}$.
  We consider only the first case, as the proof of the second is similar.
  Write
  \begin{equation*}
    \begin{aligned}
      &\sum_{i=1}^{N} \langle w_{P}; w_{P_{i}} \rangle w_{P_{i}} |I_{P_{i}}| = \Mod_{\xi_{P}} \sum_{i=1}^{N} \langle\Mod_{-\xi_{P}} w_{P};\Mod_{-\xi_{P}} w_{P_{i}} \rangle  \Mod_{-\xi_{P}}w_{P_{i}} |I_{P_{i}}|
      \\
      &=\Mod_{\xi_{P}} |I_{P}|^{-1} \sum_{i=1}^{N} \langle\1_{I_{P}};\Mod_{\xi_{P_{i}}-\xi_{P}} |I_{P_{i}}|^{-1}\1_{I_{P_{i}}} \rangle  \Mod_{\xi_{P_{i}}-\xi_{P}}\1_{I_{P_{i}}}
      \\
      &=\Mod_{\xi_{P}} |I_{P}|^{-1} \sum_{i=1}^{N} \langle\1_{I_{P}};\exp_{\xi_{P_{i}}-\xi_{P}}(x-x_{P_{i}}) |I_{P_{i}}|^{-1}\1_{I_{P_{i}}} \rangle  \exp_{\xi_{P_{i}}-\xi_{P}}(x-x_{P_{i}})\1_{I_{P_{i}}}
      \\
      &=\Mod_{\xi_{P}} |I_{P}|^{-1} \sum_{i=1}^{N} \langle\1_{I_{P}}; |I_{P_{i}}|^{-1}\1_{I_{P_{i}}} \rangle \1_{I_{P_{i}}}.
  \end{aligned}
\end{equation*}
The third identity comes from the fact that $\omega_{P}\subset\omega_{P_{i}}$ and thus $|\xi_{P_{i}}-\xi_{P}|<|I_{P}|^{-1}$, so that by \eqref{eqn:walsh-exp} it holds that 
\begin{equation*}
  \exp_{\xi_{P_{i}}-\xi_{P}}(x-x_{P_{i}}) = 1.
\end{equation*}
Since the intervals $(I_{P_{i}})_{i\in\{1,\ldots,N\}}$ partition $I_{P}$, we have
\begin{equation*}
  \sum_{i=1}^{N} \langle\1_{I_{P}}; |I_{P_{i}}|^{-1}\1_{I_{P_{i}}} \rangle \1_{I_{P_{i}}}=\1_{I_{P}},
\end{equation*}
completing the proof when $I_{P_i} \subset I_P$ for all $i$.
\end{proof}

The expression \eqref{eqn:tritile-form} of the tritile form involves multiplication of `nearby' wave packet coefficients of three separate functions.
This `nearness' of tiles is encoded by grouping triples of frequency-adjacent tiles into \emph{tritiles}.

\begin{defn}[Tritiles]\label{def:3tile}
  A \emph{tritile} is a rectangle $\mb{P} = I_{\mb{P}} \times \omega_{\mb{P}}$ of area $3$, such that the sides $I_{\mb{P}}$ and $\omega_{\mb{P}}$ are intervals.
  As with tiles, for every tritile $\mb{P}$ there are unique $x_{\mb{P}},\xi_{\mb{P}} \in \WW$ and $n\in\Z$ such that
  \begin{equation}
    \label{eq:ptile-parameters}
    \begin{aligned}
      &  I_{\mb{P}}=B_{3^n}(x_{\mb{P}}) & & [x_{\mb{P}}]_{j}=0 \text{ for } j  < n \\
      &  \omega_{\mb{P}}=B_{3^{-n+1}}(\xi_{\mb{P}}) & & [\xi_{\mb{P}}]_{j}=0 \text{ for }  j< -n+1.
    \end{aligned}
  \end{equation}
  We denote the set of all tritiles by $\tPP$.
  Every tritile $\mb{P}$ can be written in a unique way as a disjoint union of $3$ tiles with time interval $I_{\mb{P}}$; these tiles are given by
  \begin{equation}
    \label{eq:3tile-split}
    \mb{P}_{v} := I_{\mb{P}}\times\ch_{v}(\omega_{\mb{P}}), \qquad \forall v \in \{0,1,2\}.
  \end{equation}
  Conversely, for every tile $P$, there is a unique tritile $\mb{P}$ such that that $P = \mb{P}_v$ for some $v \in \{0,1,2\}$.
  This splitting of $\mb{P}$ into tiles is the \emph{horizontal splitting}; there is also a \emph{vertical splitting}
  \begin{equation}
    \label{eq:vert-split}
    \vs{\mb{P}} := \{J \times \omega_{\mb{P}} : J \in \ch(I)\}
  \end{equation}
  that we will use less often.
\end{defn}

The horizontal and vertical splittings are sketched in Figure \ref{fig:splittings}.

\begin{figure}\label{fig:splittings}

  \begin{tikzpicture}[scale=1.5]

    \draw[black] (-2,0) -- (-1,0) -- (-1,1) -- (-2,1) -- cycle;
    \node at (-1.5,-0.2) {$\mb{P}$};
    
    \draw[black] (0,0) -- (1,0) -- (1,1) -- (0,1) -- cycle;
    \draw[black] (0,0.333) -- (1,0.333);
    \draw[black] (0,0.666) -- (1,0.666);
    \node[scale=0.75][above] at (0.5,0) {$\mb{P}_0$};
    \node[scale=0.75][above] at (0.5,0.333) {$\mb{P}_1$};
    \node[scale=0.75][above] at (0.5,0.666) {$\mb{P}_2$};

    \draw[black] (2,0) -- (3,0) -- (3,1) -- (2,1) -- cycle;
    \draw[black] (2.333,0) -- (2.333,1);
    \draw[black] (2.666,0) -- (2.666,1);
    \draw[decoration={brace,mirror,raise=5pt},decorate]
    (2,0) -- node[below=6pt]{$\vs{\mb{P}}$} (3,0);
  \end{tikzpicture}

  \caption{A tritile, the horizontal splitting, and the vertical splitting.}
\end{figure}
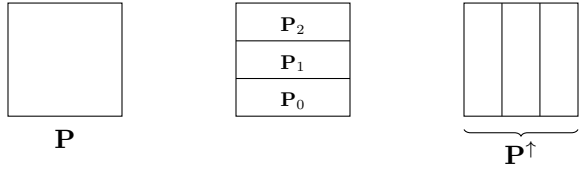

\begin{rmk}\label{rmk:tritile-functions}
  It is occasionally useful to identify the tritile $\mb{P}$ with the set of corresponding tiles $\{\mb{P}_0,\mb{P}_1,\mb{P}_2\}$, and to consider these tiles as `subtiles' of $\mb{P}$.
  Furthermore, given a Banach space $X$ and a triple-valued function on tritiles $\map{F}{\tPP}{X^3}$, we can identify $F$ with an $X$-valued function on tiles $\map{\tilde{F}}{\PP}{X}$ defined by
  \begin{equation*}
    \tilde{F}(P) = F(\mb{P})_u,
  \end{equation*}
  where $\mb{P} \in \tPP$ and $u \in \{0,1,2\}$ are uniquely determined such that $\mb{P}_u = P$.
  We will abuse notation and write $F = \tilde{F}$.
\end{rmk}

We consider $\tPP$ as being the `correct' representation of the extended Walsh phase plane, and for us it plays the role that $\R^3_+$ plays for time-frequency analysis on the real line, as explained in the introduction.

One of Fefferman's (many) innovations in his proof of Carleson's theorem was the introduction of a partial order on tiles.
Using this order one can define \emph{trees}, which represent sets of tiles that are frequency-localised at a certain `top frequency', with time restricted to a given interval.
On these subsets, time-frequency analysis is essentially reduced to Calder\'on--Zygmund theory.\footnote{See for example Proposition \ref{prop:lacunary-tree-proj}, which controls $L^p$ norms of randomised sums (the Banach-valued analogue of square functions) of projections of a function $f$ onto a tree $T$ in terms of $\|f\|_p$.}

\begin{defn}[Order and trees]\label{def:tree}
  Given two tritiles $\mb{P}$ and $\mb{P}^\prime$, we say that
  \begin{equation}
    \label{eq:ptile-order}
    \mb{P}^\prime \leq \mb{P} \qquad\text{if}\qquad I_{\mb{P}^\prime}\subseteq I_{\mb{P}} \text{ and } \omega_{\mb{P}^\prime}\supseteq \omega_{\mb{P}}.
  \end{equation}
  The \emph{tree} with top $\mb{P}$ is the collection of tritiles
  \begin{equation}
    \label{eq:tree}
    T(\mb{P}):=\{\mb{Q} \in \tPP \colon \mb{Q} \leq \mb{P}\}.
  \end{equation}
  Given a tree $T$ we denote by $\mb{P}_{T}$ the unique tritile such that $T=T(\mb{P}_{T})$. We write  $I_{T}:=I_{\mb{P}_{T}}$, $\omega_{T}:=\omega_{\mb{P}_{T}}$, $x_{T}=x_{\mb{P}_{T}}$, and $\xi_{T}=\xi_{\mb{P}_{T}}$.
  The collection of all trees is denoted by $\TT$.
  For each $u \in \{0,1,2\}$ the \emph{$u$-component} of $T$ is given by
  \begin{equation*}
    T^u := \{\mb{Q} \in T \colon \omega_T \cap \omega_{\mb{Q}_u}\neq \emptyset\},
  \end{equation*}
  so that $\mb{P}_{T}\subset T^{u}$ for all $u\in\{0,1,2\}$, and the sets $T^u \sm \{\mb{P}_T\}$ partition $T \sm \{\mb{P}_T\}$.
\end{defn}

\begin{rmk}\label{rmk:tile-tree}
  Given a tile $P$ and a tree $T$, it will be useful to write $P \in T$ to mean that $\mb{P} \in T$, where $\mb{P}$ is the unique tritile containing $P$ as a subtile (in the horizontal decomposition).
\end{rmk}

Another important class of subsets are the \emph{strips}, which consist of tiles with time restricted to a given interval, with no restriction on frequency.
These play an important role in the construction of iterated outer-$L^p$ quasinorms.
\begin{defn}[Strips]
  Given an interval $I \subset \WW$, the \emph{strip} $D=D(I)$ with \emph{top} $I$ is the collection of tritiles
  \begin{equation*}
    D(I) := \{\mb{P} \in \tPP : I_P \subset I\}.
  \end{equation*}
  Given a strip $D$ we denote by $I_{D}$ the unique interval such that $D=D(I)$. The collection of  all strips is denoted by $\DD$.
\end{defn}

Finally, we define the notion of convexity for sets of tritiles.

\begin{defn}[Convex sets]\label{defn:convex}
  A set of tritiles $\AA\subset \tPP$ is \emph{convex} if $\mb{P},\mb{P}'\in \AA$, $\mb{Q}\in \tPP$, and $\mb{P}\leq\mb{Q}\leq\mb{P}'$ imply $\mb{Q}\in\AA$. 
\end{defn}

Note that trees, strips, and their complements are convex, and that the intersection of two convex sets is convex.

\subsection{The embedding and the defect}\label{sec:defect}

Consider a Banach space $X$ and a function $\map{f}{\WW}{X}$.
Recall from the introduction the \emph{embedding} $\map{\Emb[f]}{\PP}{X}$, defined by
\begin{equation*}
  \Emb[f](P) = \langle f; w_P \rangle.
\end{equation*}
A general function $\map{F}{\PP}{X}$ cannot be realised as an `embedded function' $F = \Emb[f]$, as the wave packet coefficients $\langle f; w_P \rangle$ are not independent.
This lack of independence is codified by the relations in Lemma \ref{lem:wave-packet-interaction}.
We use these relations to construct a `defect operator', which measures how far a function $\map{F}{\PP}{X}$ is from being an embedded function.

\begin{defn}[Defect operator]
  Given a Banach space $X$ and a function $\map{F}{\PP}{X}$, the \emph{defect} $\map{\DEF F}{\PP}{X}$ is given by
  \begin{equation}\label{eq:defect}
    \DEF F (P) =F(P) -  \sum_{Q\in\vs{\mb{P}}}F(Q) \bigl\langle w_{Q}; w_{P} \bigr\rangle |I_{Q}| \qquad (P \in \PP)
  \end{equation}
  where $\mb{P} \in \tPP$ is the unique tritile containing $P$, and $\vs{\mb{P}}$ is the vertical splitting of $\mb{P}$ defined in \eqref{eq:vert-split}.

\end{defn}

The defect operator satisfies
\begin{equation}\label{eq:defect-norm-bound}
  \| \DEF F(P) \|_{X}\lesssim \| F(P) \|_{X}+ \sum_{Q \in \vs{\mb{P}}} \| F(Q) \|_{X}
\end{equation}
and, in virtue of Lemma \ref{lem:wave-packet-interaction}, if $F=\Emb[f]$ for some $\map{f}{\WW}{X}$ then $\DEF F=0$.

In the following proposition, we show how a function $\map{F}{\PP}{X}$ can be decomposed as the sum of an embedded function and its defect.

\begin{prop}[Function reconstruction]\label{prop:func-reconstruction}
  Let $T$ be a tree, and let $P$ be a tile with $P \in T$ (recall from Remark \ref{rmk:tile-tree} that this means $\mb{P} \in T$, where $\mb{P}$ is the unique tritile with $P \in \mb{P}$).
  Then for all\ $N\in\N$ it holds that
  \begin{equation}\label{eq:func-reconstruction}
  \begin{aligned}
    F(P) =  \DEF F(P) & +
    \biggl\langle
      \sum_{\substack{Q\in T \\ |I_{P}|>|I_Q|\geq3^{-N}|I_{P}| }}
      \DEF F(Q) \,w_{Q} |I_Q|\,
      \,;\;
      w_{P}
    \biggr\rangle
    \\
    &
    +
    \biggl\langle
      \sum_{\substack{Q\in T \\ |I_Q|=3^{-(N+1)}|I_P|}}
      F(Q)w_{Q}|I_{Q}|
      \,;\;
      w_{P}
    \biggr\rangle.
  \end{aligned}
\end{equation}
\end{prop} 

\begin{proof}
  We induct on $N \in \N$.
  If $N=0$ then the result follows immediately by definition of $\DEF F$, as the first sum is empty and the condition in the second sum is a rewriting of the condition $Q \in \vs{\mb{P}}$.

  Let us show that if \eqref{eq:func-reconstruction} holds for $N$ then it also holds for $N+1$.
  Apply the result with $N=0$ to each tile in the second sum to obtain 
  \begin{align*}
    &\sum_{\substack{Q\in T \\ |I_{Q}|=3^{-(N+1)}|I_{P}|}} F(Q) \langle  w_{Q}|I_{Q}|;w_P \rangle \\
    &=\sum_{\substack{Q\in T \\ |I_Q|=3^{-(N+1)}|I_P| }} \bigg(
      \DEF F(Q) + 
      \sum_{\substack{Q'\in T \\ |I_{Q'}|=3^{-(N+2)}|I_P|}} F(Q') \langle
        w_{Q'}|I_{Q'}|
        ;
        w_{Q} \rangle
     \bigg) \langle w_Q|I_Q|;w_P \rangle
    \\
    &= \mb{I} +
    \biggl\langle
      \sum_{\substack{Q'\in T \\ |I_{Q'}|=3^{-(N+2)}|I_P|}} F(Q') w_{Q'}|I_{Q'}|
    ;
    \sum_{\substack{Q\in T \\ |I_Q|=3^{-(N+1)}|I_P| }} \langle w_P;w_Q \rangle w_{Q}|I_Q| \biggr\rangle
     \\
    &=
      \mb{I}
    +
    \biggl\langle
      \sum_{\substack{Q \in T \\ |I_{Q'}|=3^{-(N+2)}|I_P|}}
      \hspace{-1.5em}
      F(Q')w_{Q'}|I_{Q'}|
      ;
      w_P 
    \biggr\rangle  
  \end{align*}
  where
  \begin{equation*}
    \mb{I} := \sum_{\substack{Q\in T \\ |I_Q|=3^{-(N+1)}|I_P| }} 
      \DEF F(Q) \langle w_Q|I_Q|;w_P \rangle;
  \end{equation*}
  where the last identity holds since the tiles $Q \in T$ with $|I_Q| = 3^{-(N+1)} |I_P|$ are disjoint and cover $P$. Plugging this into (\ref{eq:func-reconstruction}) for $N$ gives the statement for $N+1$ as required.
\end{proof}

\begin{rmk}
  We have remarked that $\DEF F=0$ if $F=\Emb[f]$ for some $\map{f}{\WW}{X}$.
  A similar result is true if $F$ is not precisely an embedded function, but rather a `cut-off' embedded function; for this result we need to think in terms of tritiles rather than tiles.
  If $F=\Emb[f]$ and $\AA\subset \tPP$, then $\DEF (\1_{\AA}F) (\mb{P}) \neq 0$ only if $P$ happens to be on the ``boundary'' of the set $\AA$; that is, if $\mb{P}\in\AA$ and there exists $\mb{Q}\leq \mb{P}$ with $|I_{\mb{Q}}|=|I_{\mb{P}}|/3$ such that $\mb{Q}\notin \AA$, or if $\mb{P}\notin \AA$ and there exists $\mb{Q}\leq \mb{P}$ with $|I_{\mb{Q}}|=|I_{\mb{P}}|/3$ such that $\mb{Q}\in\AA$.
  A crucial observation is that if $\AA$ is convex, then for any fixed $x\in\WW$ there exist at most two tritiles $\mb{P}$ on the boundary of $\AA$ with $x\in I_{\mb{P}}$.
  \end{rmk}


 \section{Analysis in Banach spaces}
 \label{sec:vvwpa}
 The harmonic analysis of functions $\map{f}{\WW}{X}$ valued in a Banach space $X$ exhibits phenomena that are not present in the scalar case $X = \CC$.
Generally techniques that work for scalar-valued functions require geometric assumptions on $X$ in order to have $X$-valued extensions.
The most famous of these geometric assumptions is the UMD (Unconditional Martingale Differences) property, which we discuss in Section \ref{sec:UMD}.
We will also require the $q$-Hilbertian property (also referred to as the $\theta$-Hilbertian property in the literature).
Before discussing these geometric assumptions we give a short introduction to Rademacher sums, a crucial tool in Banach-valued analysis without which not much can be said.

A relatively complete introduction to Banach-valued analysis is the incomplete series \cite{HNVW16,HNVW17}.
The reader will benefit from having a copy of these references at hand while reading this paper.

\subsection{Rademacher sums}\label{sec:r-sums}

A great deal of scalar-valued harmonic analysis is connected with square functions; that is, functions of the form
\begin{equation*}
   t \mapsto \bigg( \sum_{n=1}^N |f_n(t)|^2 \bigg)^{1/2} \qquad \forall t \in \R
\end{equation*}
where $(f_n)_{n \in \{1,\ldots,N\}}$ is a sequence of $\CC$-valued functions on $\RR$ (for example).
If $X$ is a Banach lattice (or in particular, a function space), then for all finite sequences $(x_n)_{n=1}^N$ in $X$ one can make sense of the quantity
\begin{equation*}
  \bigg( \sum_{n=1}^N |x_n|^2 \bigg)^{1/2} \in X
\end{equation*}
as an element of $X$.
However, for general Banach spaces $X$, this is not possible.
The correct $X$-valued analogue of a square function is a \emph{Rademacher sum}, which is a quantity of the form
\begin{equation*}
  \E \bigg\| \sum_{n=1}^N \varepsilon_n x_n \bigg\|_X := \int_\Omega \bigg\| \sum_{n=1}^N \varepsilon_n(\omega) x_n \bigg\|_X \, \dd \omega ,
\end{equation*}
where $(x_n)_{n=1}^N$ is a finite sequence in $X$, and where $(\varepsilon_n)_{n=1}^N$ is a sequence of independent Rademacher variables on some probability space $\Omega$, i.e. random variables taking the values $\pm 1$ with probability $1/2$.
When $X$ is a Banach lattice with finite cotype (for example, if $X = L^p(\Xi)$ for some $\sigma$-finite measure space $\Xi$, with $p \in [1,\infty)$), then Rademacher sums are equivalent to norms of square functions; that is,
\begin{equation}\label{eqn:khintchine-maurey}
  \E \bigg\| \sum_{n=1}^N \varepsilon_n x_n \bigg\|_{X} \simeq \bigg\| \bigg( \sum_{n=1}^N |x_n|^2 \bigg)^{1/2} \bigg\|_{X}
\end{equation}
for all finite sequences $(x_n)_{n \in \{1,\ldots,N\}}$ in $X$. This is the Khintchine--Maurey theorem \cite[Theorem 7.2.13]{HNVW17}.
In this paper we do not work with Banach lattices, so (other than this paragraph) we do not discuss square functions; only Rademacher sums.

Here we mention two particularly important results that allow us to manipulate Rademacher sums.
The first lets us replace the expectation in a Rademacher sum with an $L^p$-expectation for any $p \in (0,\infty)$; the second lets us pull out bounded scalar coefficients in a Rademacher sum.
We will use these results throughout the paper, often without mention.
For proofs see \cite[Theorems 6.2.4 and 6.1.13]{HNVW17}.

\begin{thm}[Kahane--Khintchine]\label{thm:kahane-khintchine}
  Let $X$ be a Banach space.  For all finite sequences $(x_n)_{n=1}^N$ in $X$ and all $p \in (0,\infty)$, we have the equivalence
  \begin{equation}\label{eqn:kahane-khintchine}
   \E \bigg\| \sum_{n=1}^N \varepsilon_n x_n \bigg\|_X \simeq_{p} \bigg(\E \bigg\| \sum_{n=1}^N \varepsilon_n(\omega) x_n \bigg\|_X^p \bigg)^{1/p}
 \end{equation}
 with implicit constant independent of $N$.
\end{thm}

\begin{thm}[Kahane's contraction principle]\label{thm:contraction-principle}
  Let $X$ be a Banach space. For all finite sequences $(x_n)_{n=1}^N$ in $X$ and $(a_n)_{n=1}^N$ in $\CC$, we have
  \begin{equation}\label{eq:contraction-principle}
    \E \bigg\| \sum_{n=1}^N \varepsilon_n a_n x_n \bigg\|_X \lesssim  \|a\|_\infty \E \bigg\| \sum_{n=1}^N \varepsilon_n x_n \bigg\|_X,
  \end{equation}
  with implicit constant independent of $N$.
\end{thm}

\subsection{The UMD property}\label{sec:UMD}

As already mentioned, the most important of our geometric assumptions is the UMD property.
It is natural to assume this property when doing Banach-valued harmonic analysis, as a Banach space $X$ is UMD if and only if the Hilbert transform extends to a bounded operator on $L^p(\RR;X)$ for all $p \in (1,\infty)$ \cite{jB83,dB83}. The classical reflexive function spaces, for example $L^p$-spaces, Sobolev spaces, and Triebel--Lizorkin and Besov spaces, are all UMD. However, there are also important UMD spaces that are not function spaces (or not even Banach lattices); in particular, non-commutative $L^p$-spaces, including the Schatten classes $\mc{C}^p$ (see \cite[Chapter 14]{gP16} and \cite[Appendix D]{HNVW16}). For more exposition on UMD spaces see for example \cite{dB01, gP16, HNVW16}.
We recall one possible definition of the UMD property in terms of Haar decompositions.

For every dyadic interval $J=[m2^{n},(m+1)2^{n}) \subset \RR$, $n,m\in\Z$, define the $L^1$-normalised Haar function
\begin{equation}\label{eq:haar}
  h_{J}:=|J|^{-1}( \1_{J_{0}}-\1_{J_{1}}),
\end{equation}
where $J_0$ and $J_1$ are the left and right halves of $J=J_{0} \cup J_{1}$, i.e.
\begin{equation}
  J_{0} := [m2^{n},(2m+1)2^{n-1}), \; J_{1} := [(2m+1)2^{n-1},(m+1)2^{n} ).
\end{equation}
It is straightforward to see that $\langle h_{J};h_{J'}\rangle=0$ unless $J=J'$, and thus
\begin{equation}
  \label{eq:haar-transform}
    \bigg\| \sum_{J\subset [0,1)}  a_{J} \langle f ; h_{J} \rangle h_{J} |J| \bigg\|_{L^2(I;\CC)} \leq \| f \|_{L^{2}([0,1);\CC)} \qquad \forall f \in L^2([0,1);\CC)
 \end{equation}
 for any finitely-supported sequence of signs $a_J \in \{-1,1\}$, where the sum is over all dyadic intervals $J \subset [0,1)$.
 When $L^2$ is replaced with $L^p$ for some $p \in (1,\infty)$, the estimate \eqref{eq:haar-transform} still holds, with a constant depending on $p$ (although naturally the proof above, being reliant on orthogonality, does not extend to $p \neq 2$).
 This motivates the following definition.
 
\begin{defn}
  A Banach space $X$ has the \emph{UMD property} if there exists $p \in (1,\infty)$ such that for any $f\in L^{p}([0,1),X)$ and any finitely-supported sequence $(a_{J})_{J \subset [0,1)}$ of signs, it holds that 
  \begin{equation}
      \label{eq:haar-transform-UMD}
\bigg\| \sum_{J\subset [0,1)} a_{J}\langle f ; h_{J}\rangle h_{J} |J| \bigg\|_{L^p([0,1);X)} \lesssim \| f \|_{L^{p}([0,1),X)}
  \end{equation}
  where the sum is over all dyadic intervals $J\subset [0,1)$.
\end{defn}

If \eqref{eq:haar-transform-UMD} holds for one $p \in (1,\infty)$, then it holds for all $p \in (1,\infty)$ (with a different constant) and with $[0,1)$ replaced by any dyadic interval (see \cite[Theorems 4.2.7 and 4.2.12]{HNVW16}.

The Haar functions are in fact $2$-Walsh wave packets associated to $T^1\bigl(B_{1}(0)^{2}\bigr)$, so the bound \eqref{eq:haar-transform-UMD} can be interpreted as unconditionality of a tree projection operator. In the $3$-Walsh case, we use the following randomised version of (\ref{eq:haar-transform-UMD}); the proof is a bit harder than the $2$-Walsh case because the tree projections cannot be directly related to martingale transforms.
The idea is to reduce to the tree $T(B_1(0)^2)$ by modulation, translation, and dilation, and then to reduce matters to a result of Cl\'ement et al. \cite{CPSW00} which has already done the hard work of relating $3$-Walsh--Fourier projections to martingale transforms.

\begin{prop}\label{prop:lacunary-tree-proj}
  Let $p \in (1,\infty)$ and $X$ be a UMD Banach space.
  Then for all trees $T$ and all $f \in L^p(I_T;X)$ we have
  \begin{equation}\label{eqn:RUC-E3}
    \E \bigg\| \sum_{\substack{u,v=0\\u\neq v}}^{2}\sum_{\mb{P}\in T^{u}} \varepsilon_{\mb{P}_{v}} \langle f; w_{\mb{P}_v} \rangle w_{\mb{P}_v} |I_{\mb{P}}| \bigg\|_{L^p(I_{T};X)} \lesssim  \|f\|_{L^p(I_T;X)}.
  \end{equation}
\end{prop}

\begin{proof}
  First we reduce to consideration of the tree $T_1 := T(B_1(0)^2)$.
  Fix an arbitrary tree $T$.
  Define the `lacunary tiles' associated with $T$ to be the set of tiles
  \begin{equation*}
    T^{\mathrm{lac}}:=\bigcup_{\substack{u,v=0\\u\neq v}}^{2}\{\mb{P}_{v} : \mb{P}\in T^{u} \} =     \bigcup_{\substack{v=0}}^{2}\{\mb{P}_{v} : \mb{P}\in T, \, \xi_{T}\notin \omega_{\mb{P}_{v}}\}.
  \end{equation*}
  The lacunary tiles associated with $T$ can be related to those associated with $T_1$ by the relation
  \begin{equation*}
    \Dil_{|I_{T}|^{-1}}\Tr_{-x_{T}}\Mod_{-\xi_{T}} T^{\mathrm{lac}} =T_1^{\mathrm{lac}},  
  \end{equation*}
  with dilation, translation, and modulation operators acting on tiles as in Remark \ref{rmk:wave-packets-tiles}.
  Applying these operators to the wave packets appearing in \eqref{eqn:RUC-E3} we obtain 
  \begin{align*}
    & \Dil_{|I_{T}|^{-1}}\Tr_{-x_{T}}\Mod_{-\xi_{T}} \sum_{\substack{u,v=0\\u\neq v}}^{2}\sum_{\mb{P}\in T^{u}} \varepsilon_{\mb{P}_{v}} \langle f; w_{\mb{P}_v} \rangle w_{\mb{P}_v} |I_{\mb{P}}|
    \\
    &= \sum_{P\in T^{\mathrm{lac}}} \varepsilon_{P} \langle f; w_{P} \rangle \, \Dil_{|I_{T}|^{-1}}\Tr_{-x_{T}}\Mod_{-\xi_{T}}  w_{P} \,|I_{P}|
    \\
    &
      = \sum_{ P\in T_1^{\mathrm{lac}} } a_{P}\varepsilon_{P} \langle \Dil_{|I_{T}|^{-1}}\Tr_{-x_{T}}\Mod_{-\xi_{T}}f; w_{P} \rangle   \,w_{P}\, |I_{P}|
  \end{align*}
  for some unimodular constants $a_{P}$.
  Supposing that \eqref{eqn:RUC-E3} holds for $T_1$, the contraction principle (\ref{eq:contraction-principle}) yields that 
  \begin{align*}
    &|I_{T}|^{-1/p'}\E \bigg\| \sum_{\substack{u,v=0\\u\neq v}}^{2}\sum_{\mb{P}\in T^{u}} \varepsilon_{\mb{P}_{v}} \langle f; w_{\mb{P}_v} \rangle w_{\mb{P}_v} |I_{\mb{P}}| \bigg\|_{L^p(I_{T};X)}
    \\
    &= \E \bigg\| \sum_{\substack{u,v=0\\u\neq v}}^{2}\sum_{\mb{P}\in T_1^{u}} a_{P_{v}}\varepsilon_{\mb{P}_{v}} \langle \Dil_{|I_{T}|^{-1}}\Tr_{-x_{T}}\Mod_{-\xi_{T}} f; w_{\mb{P}_v} \rangle w_{\mb{P}_v} |I_{\mb{P}}| \bigg\|_{L^p(B_{1}(0);X)}
    \\
    &\lesssim \|   \Dil_{|I_{T}|^{-1}}\Tr_{-x_{T}}\Mod_{-\xi_{T}} f \|_{L^p(B_{1}(0);X)}
      =|I_{T}|^{-1/p'}\|  f \|_{L^p(I_{T};X)}.
  \end{align*}
  
  Thus it suffices to show \eqref{eqn:RUC-E3} for the tree $T = T_1$.
  For this tree, only the $0$-part is nontrivial; i.e. $T_1=T_1^{0}$.
  Let us show the bound restricted to the summand corresponding to $v=1$; the bound for the $v=2$ summand is shown in the same way, and one combines these summands using the triangle inequality.
  Using the Kahane--Khintchine inequality (Theorem \ref{thm:kahane-khintchine}) and Fubini one has
  \begin{align*}
    &\Biggl(  \E \bigg\| \sum_{\mb{P}\in T^{0}} \varepsilon_{\mb{P}_{1}} \langle f; w_{\mb{P}_1} \rangle w_{\mb{P}_1} |I_{\mb{P}}| \bigg\|_{L^p(B_{1}(0);X)}\Biggr)^{p} \\
    &\simeq \E \bigg\| \sum_{\mb{P}\in T^{0}} \varepsilon_{\mb{P}_{1}} \langle f; w_{\mb{P}_1} \rangle w_{\mb{P}_1} |I_{\mb{P}}| \bigg\|_{L^p(B_{1}(0);X)}^{p}
    \\
    &
      = \int_{B_{1}(0)} \int_\Omega \bigg\| \sum_{\mb{P}\in T^{0}} \varepsilon_{\mb{P}_{1}}(\omega) \langle f; w_{\mb{P}_1} \rangle w_{\mb{P}_1}(x) |I_{\mb{P}}| \bigg\|_{X}^{p} \, \dd \omega \, \dd x
    \\
    &= \int_{B_{1}(0)} \int_\Omega \bigg\| \sum_{\mb{P}\in T^{0}} \varepsilon_{|I_{\mb{P}}|}(\omega) \langle f; w_{\mb{P}_1} \rangle w_{\mb{P}_1}(x) |I_{\mb{P}}| \bigg\|_{X}^{p} \, \dd \omega \, \dd x
    \\
    &\simeq  \E \bigg\| \sum_{n=0}^{\infty}\varepsilon_{n}\hspace{-1em}\sum_{\substack{\mb{P}\in T^{0}\\ |I_{\mb{P}}|=3^{-n}}}  \langle f; w_{\mb{P}_1} \rangle w_{\mb{P}_1} |I_{\mb{P}}| \bigg\|_{L^p(B_{1}(0);X)}^{p}.
  \end{align*}
  In reindexing the Rademacher variables we used that for each $x \in B_1(0)$ the tiles $\mb{P} \in T^0$ for which $w_{\mb{P}_1}(x) \neq 0$ are in bijective correspondence with the scales $\{|I_\mb{P}| : \mb{P} \in T_0\}$, and thus the two sets of Rademacher variables
  \begin{equation*}
    \{\varepsilon_{\mb{P}_1} : \mb{P} \in T_0, w_{\mb{P}_1}(x) \neq 0\}, \quad \{\varepsilon_{|I_{\mb{P}}|} : \mb{P} \in T_0\}
\end{equation*}
are equally distributed.

For each $n \in \N$, let $S_{n}$ denote the Walsh--Fourier projection onto the interval $B_{3^{-n}}(3^{-n})=\{\xi \in \WW \colon \xi_{-k}=\delta_{-n}(k)\}$, so that
\begin{equation*}
   \E \bigg\| \sum_{\mb{P}\in T^{0}} \varepsilon_{\mb{P}_{1}} \langle f; w_{\mb{P}_1} \rangle w_{\mb{P}_1} |I_{\mb{P}}| \bigg\|_{L^p(B_{1}(0);X)} \simeq \E  \bigg\|\sum_{n=0}^{\infty} \varepsilon_{n} S_{n}f \bigg\|_{L^p(B_{1}(0);X)}.
 \end{equation*}
 To bound this quantity, we use a result of Cl\'ement et al. \cite[Corollary 4.4]{CPSW00}; since $X$ is UMD, this result implies
 \begin{equation}\label{eqn:our-notation}
    \E \bigg\| \sum_{n \geq 1} \varepsilon_n S_{n} f\bigg\|_p \lesssim \|f\|_{L^p(B_1(0);X)} \qquad \forall f \in L^p(B_1(0);X)
  \end{equation}
  and completes the proof.\footnote{
    We briefly show how to deduce \eqref{eqn:our-notation} from \cite[Corollary 4.4]{CPSW00}, assuming familiarity with the notation of \cite[Section 4]{CPSW00}.
    The interval $B_{3^{-n}}(3^{-n})$ can be identified with the set
    \begin{equation*}
      \{\mb{n} \in \WW : \mb{d}_{(n,1)} \leq \mb{n} < \mb{d}_{(n,2)}\},
    \end{equation*}
    and thus the Walsh--Fourier projection $S_{n}$ can be identified with the projection $\Delta_{(n,1)}$.
    Since $X$ is UMD, \cite[Corollary 4.4]{CPSW00} says that the set $\{\Delta_{(n,v)} : (n,v) \in \NN \times \{1,2\}\}$ is an unconditional Schauder decomposition of $L^p(\WW;X)$, and this implies \eqref{eqn:our-notation} by the contraction principle.}
\end{proof}

\begin{rmk}
  With additional work, one could improve the randomised estimate \eqref{eqn:RUC-E3} to full unconditionality (i.e. replacing the Rademacher variables with an arbitrary deterministic choice of signs) by working through \cite[Section 4]{CPSW00}, modifying the martingale difference sequence to take into account orthogonal wave packets at the same scale, as in the proof of unconditionality of the Haar decomposition (see \cite[Theorem 4.2.13]{HNVW16}).
  Since we only need the randomised estimate, we leave this to the hypothetical interested reader.
\end{rmk}

\subsection{$r$-Hilbertian spaces}

For $p,q \in [1,\infty]$ and $\theta \in [0,1]$, we let $[p,q]_\theta \in [1,\infty]$ be the number defined by the relation
\begin{equation*}
  \frac{1}{[p,q]_\theta} = \frac{1-\theta}{p} + \frac{\theta}{q}.
\end{equation*}

\begin{defn}
  Let $r \in [2,\infty)$.
  We say that a Banach space $X$ is \emph{$r$-Hilbertian} if there exists a Hilbert space $H$ and a Banach space $Y$, such that $(H,Y)$ is an interpolation couple, and such that $X$ is isomorphic to the complex interpolation space $[H,Y]_\theta$, with $[2,\infty]_\theta = r$.
\end{defn}

\begin{rmk}
  In \cite{gP16}, $r$-Hilbertian spaces are referred to as $\theta$-Hilbertian.
  In our computations the parameter $r$ plays a more important role, so we prefer to use our terminology.
\end{rmk}

For an introduction to interpolation spaces, see for example \cite{BL78} or \cite[Appendix C]{HNVW16}. Note that if $X$ is $r$-Hilbertian, then $X$ is $s$-Hilbertian for all $s > r$. Every $L^p$-space with $p \in [2,\infty)$, either classical or non-commutative, is $p$-Hilbertian: to see this, note that $L^p = [L^2,L^\infty]_\theta$ with $[2,\infty]_\theta = p$. By the same argument, replacing $L^\infty$ with $L^1$, $L^p$ is $p^\prime$-Hilbertian when $p \in (1,2]$.

$r$-Hilbertian spaces enjoy the following `$r$-orthogonality' of wave packet coefficients, which should be compared to the notions of tile-type and quartile-type in \cite{HL13, HL18, HLP13, HLP14}. It should be noted that this is the only consequence of the $r$-Hilbertian property that we actually use. Thus one could isolate this estimate as a geometric assumption, perhaps called `Walsh tile-type $r$' (although that name is already taken).
However, we do not know how to establish the property without assuming the $r$-Hilbertian property, so we choose not to make this definition.

\begin{prop}[Walsh tile-type]\label{prop:tile-orthogonality}
  Let $X$ be $r$-Hilbertian, then 
  \begin{equation}\label{eqn:tile-orthogonality}
    \biggl( \sum_{P \in A} \|\langle f; w_P \rangle \|_X^r \, |I_{P}|\biggr)^{1/r} \lesssim \|f\|_{L^r(\WW;X)} \qquad \forall f \in L^r(\WW;X),
  \end{equation}
  for any  finite collection $A\subset\PP$ of pairwise disjoint tiles,  with implicit constant independent of $A$.
\end{prop}

\begin{proof}
  Suppose $X$ is isomorphic to $[H,Y]_\theta$, where $H$ is a Hilbert space, $Y$ is a Banach space, and $[2,\infty]_\theta = r$. Let $\mathring{L}^\infty(\WW;Y)$ denote the closure of the Schwartz functions $\Sch(\WW;Y)$ in $L^\infty(\WW;Y)$.  A straightforward estimate yields
  \begin{equation*}
   \sup_{P \in A} \|\langle f; w_P \rangle\|_{Y}\leq \|f\|_{L^\infty(\WW;Y)} \qquad \forall f \in \mathring{L}^\infty(\WW;Y),
  \end{equation*}
  while Plancherel's theorem yields
  \begin{equation*}
    \biggl( \sum_{P \in A} \|\langle f; w_P \rangle\|_{H}^{2}\,|I_{P}|\biggr)^{1/2} \leq \|f\|_{L^2(\WW;H)} \qquad \forall f \in L^2(\WW;H). 
  \end{equation*}
  The desired inequality follows by complex interpolation, with all sequence spaces on $A$ weighted by $P \mapsto |I_P|$, 
  \begin{equation*}
    \ell^r(A;X)
    \cong \ell^r(A;[H,Y]_\theta)
    = [\ell^2(A;H), \ell^\infty(A;Y)]_\theta
  \end{equation*}
  (see \cite[\textsection 1.18.1, Remark 2]{hT78} for the equality at the end, and \cite[\textsection 1.18.4, Remark 3]{hT78} for interpolation between $L^2$ and $\mathring{L}^\infty$). 
\end{proof}

\begin{rmk}\label{rmk:Qiu}
  It is natural to suspect that if a Banach space $X$ is $r$-Hilbertian for some $r < \infty$, then it must be UMD.  This is false; a counterexample is given by Qiu's construction (see \cite[\textsection 4.3.c]{HNVW16} and \cite{yQ12}).\footnote{We thank Mark Veraar for pointing this example out to us.}  For all $r \in [2,\infty]$, and $k \in \N$, inductively define spaces
  \begin{equation*}
    X_0^r := \ell^2_2(\ell^r_2), \qquad X_{k+1}^r := \ell_2^2(\ell_2^r(X_{k+1}))
  \end{equation*}
  (here $\ell^r_2(Y) := \ell^r(\{0,1\};Y)$, where $\{0,1\}$ is equipped with counting measure).
  Then set $X^r := \oplus^r_{n \in \N} X_k^r$. For all $r \neq 2$, $X^r$ is not UMD, while $X^r = [X^2,X^\infty]_\theta$ is $r$-Hilbertian.
\end{rmk}


 \section{Outer-$L^p$ spaces}
 \label{sec:outermeasures}
 In this section we introduce outer structures and their associated outer-$L^p$ quasinorms.
Roughly speaking, an outer structure on a topological space consists of an outer measure on the space, a Banach space $X$, and a \emph{size} on $X$-valued functions on the topological space.
Currently the standard references on this topic are the initial work by Do and Thiele \cite{DT15}, and the first Banach-valued implementation by Di Plinio and Ou \cite{DPO18}.
However, the outer-$L^p$ concept is still quite new, and the terminology and definitions are not fixed.
Our interpretation of the theory differs slightly (but not fundamentally) from what appears in the literature.
In Sections \ref{sec:particular-outer} and \ref{sec:size-dom} we analyse particular outer structures that are relevant to our problem.

\subsection{Initial definitions}

For a topological space $\XX$ we let $\Bor(\XX)$ denote the $\sigma$-algebra of Borel sets in $\XX$, and for a Banach space $X$ we let $\Bor(\XX;X)$ denote the set of strongly Borel measurable functions $\XX \to X$.
Recall that a \emph{Polish space} is a topological space that is homeomorphic to a complete separable metric space.
This is a technical assumption that will ultimately play no role in this paper, as we only really care about the countable space $\tPP$ with the discrete topology.

\begin{defn}[Outer structure]\label{defn:outer-measure-structure}
  Let $\XX$ be a Polish space.
  An \emph{outer structure} on $\XX$, or simply an \emph{outer structure}, consists of the following data:
  \begin{itemize}
  \item a collection $\EE \subset \Bor(\XX)$ of \emph{generating sets}, 
  \item a function $\map{\sigma}{\EE}{[0,\infty)}$, called the \emph{premeasure},
  \item a Banach space $X$,
  \item an \emph{$X$-size} (or simply a \emph{size}) $S$ on $(\XX,\EE)$; that is, a family of maps indexed by $E\in\EE$
    \begin{equation*}
      \Bor(\XX;X) : F \mapsto \| F \|_{S(E)} \in [0,\infty] \qquad \forall E \in \EE
    \end{equation*}
    such that there exists a constant $C \geq 1$ satisfying the following properties for all $E \in \EE$ and $F,G \in \Bor(\XX;X)$:
    \begin{description}
    \item[unconditionality] $\| \1_{A} F \|_{S(E)} \leq C \| F \|_{S(E)}$ for all
      \begin{equation*}
        A\in\EE^{\cup}=\Bigl\{ A \colon A=\bigcup_{n\in\N}E_{n} \text{ with } E_{n}\in\EE \Bigr\}.
      \end{equation*}
    \item[homogeneity] $\| \lambda F \|_{S(E)} = |\lambda| \| F \|_{S(E)}$ for all $\lambda\in\CC$;
    \item[quasi-triangle inequality] $\| F+G \|_{S(E)} \leq C(\| F \|_{S(E)} + \| G \|_{S(E)})$;
    \item[nondegeneracy] $\| F \|_{S(E)}=0$ for all $E\in\EE$ if and only if $F=0$.
    \end{description}
    That is, the maps $\|\cdot\|_{S(E)}$ are (possibly infinite) quasinorms on $E$, with quasinorm constant uniformly bounded in $E \in \EE$, and with an additional unconditionality property.\footnote{The unconditionality property can be interpreted as a monotonicity property when $X = \CC$, or more generally when $X$ is a Banach lattice.}
  \end{itemize}

  Given an outer structure on $\XX$ as above, we define the induced outer measure $\map{\sigma}{\mc{P}(\XX)}{[0,\infty]}$ (which we denote by the same letter as the premeasure) by
  \begin{equation*}
    \sigma(A) :=  \inf \Bigl\{\sum_{E\in\mb{E}\subset \EE } \sigma(E) : \bigcup_{E\in \mb{E}} E \supset A \Bigr\}
    \qquad \forall A \subset \XX
  \end{equation*}
  where the infimum is taken over all countable covers $\mb{E}$ of $A$ by generating sets.
  For all $f \in \Bor(\XX;X)$ we define $\| f \|_S := \sup_{E \in \EE} \| f \|_{S(E)}$, and for all $\lambda > 0$ we define the outer superlevel measure
  \begin{equation*}
    \sigma(\| f \|_{S} > \lambda) := \inf\{\sigma(A) : A \subset \XX, \| \1_{\XX \sm A}f \|_{S}  \leq \lambda\}.
  \end{equation*}
\end{defn}

Different choices of sizes lead to fundamentally different outer structures, even when the outer measure and the Banach space remain fixed.
Thus we consider the size (and the underlying Banach space) as a component of the outer structure.

To each outer structure is associated a family of quasinorms, defined in a way that mimics the so-called layer cake representation of the $L^p$ norm.

\begin{defn}[Outer-$L^{p}$ quasinorms]
  Let $\XX$ be a Polish space, and let $(\EE,\sigma,X,S)$ be an outer structure on $\XX$.
  For all $p \in (0,\infty)$ we define the \emph{outer-$L^p$ quasinorms} and \emph{weak outer-$L^p$ quasinorms} of a function $f\in\Bor(\XX;X)$ by setting
  \begin{align*}
    &  \|F\|_{L_{\sigma}^{p} S} := \bigg( \int_{0}^{\infty} p\lambda^{p-1} \sigma(\| F \|_{S} > \lambda) \, \dd \lambda \bigg)^{1/p}  &&\forall p\in(0,\infty), \\
    &  \|F\|_{L_{\sigma}^{p,\infty} S} := \sup_{\lambda > 0} \lambda \,\sigma(\| F \|_{S} > \lambda)^{1/p} &&\forall p\in(0,\infty), \\
    &\|F\|_{L_{\sigma}^{\infty} S} := \| F \|_{S}. 
  \end{align*}
  It is straightforward to check that these are indeed quasinorms. 
\end{defn}

A H\"older-type inequality holds for outer-$L^p$ spaces defined with respect to different sizes provided that it holds in a certain sense for the sizes themselves. The proof below is a straightforward extension of that of \cite[Proposition 3.4]{DT15}.

\begin{prop}[Outer Hölder inequality]\label{prop:outer-holder}
  Let $\XX$ be a Polish space.
  For each $u \in \{0,1,2\}$ let $(\EE,\sigma,X_u,S_u)$ be an outer structure on $\XX$, and let $(\EE,\sigma,X,S)$ be another outer structure on $\XX$.
  Note that all these outer structures have the same generating sets and premeasure.
  Let $\map{\Pi}{X_0 \times X_1 \times X_2}{X}$ be a bounded trilinear map, and suppose that the size-H\"older inequality
  \begin{equation}
    \label{eq:abstract-size-holder}
    \| \Pi(F_{0},F_{1},F_{2}) \|_{S} \lesssim \prod_{u=0}^{2} \| F_{u} \|_{S_{u}} \qquad \forall F_u \in \Bor(\XX;X_u)
  \end{equation}
  holds.
  Then for all $p_u \in [1,\infty]$ we have the outer H\"older inequality
    \begin{equation}
    \label{eq:abstract-Lp-holder}
    \| \Pi(F_{0},F_{1},F_{2}) \|_{L_\sigma^{p} S} \lesssim_{p_0,p_1,p_2} \prod_{u=0}^{2} \| F_{u} \|_{L_{\sigma}^{p_{u}} S_{u}} \qquad \forall F_u \in \Bor(\XX;X_u)
  \end{equation}
  with $p^{-1}=\sum_{u=0}^{2}p_{u}^{-1}$.
\end{prop}

\begin{proof}
  Assume that the factors on the right hand side of \eqref{eq:abstract-Lp-holder} are finite and non-zero, for otherwise there is nothing to prove.
  By homogeneity we may assume that $\| F_{u} \|_{L^{p_{u}}_{\sigma}S_{u}}^{p_{u}} =1$ for each $u$.
  For each $u\in\{0,1,2\}$  and $n\in\Z$ let $A_{n}^{u}\subset\XX$ be such that
  \begin{equation*}
    \sum_{n\in\Z} 2^{n} \sigma(A_{n}^{u}) \lesssim 1 \qquad \| \1_{\XX\setminus A_{n}^{u}} F_{u}\|_{S_{u}}\lesssim2^{n/p_{u}}.
  \end{equation*}
  We may assume that $A_{n}^{u}\subset A_{n-1}^{u}$ by considering $\tilde{A}_{n}^{u}=\bigcup_{k\geq n} A_{n}^{u}$ and noticing that $\tilde{A}_{n}^{u}$ satisfies the conditions above.
  Let $A_{n}=\bigcup_{u=0}^{2}A_{n}^{u}$.
  Then it holds that
  \begin{equation*}
    \sum_{n\in\Z} 2^{n} \sigma(A_{n}) \lesssim \sum_{n\in\Z} \sum_{u=0}^{2}2^{n} \sigma(A_{n}^{u})\lesssim 1,
  \end{equation*}
  and from \eqref{eq:abstract-size-holder} it follows that
  \begin{equation*}
    \| \1_{\XX\setminus A_{n}} \Pi(F_{0},F_{1},F_{2})\|_{S}=
    \| \Pi( \1_{\XX\setminus A_{n}}F_{0}, \1_{\XX\setminus A_{n}}F_{1}, \1_{\XX\setminus A_{n}}F_{2})\|_{S}\lesssim2^{n/p},
  \end{equation*}
  which concludes the proof.
\end{proof}

It is possible to control classical $L^1$ norms by outer-$L^1$ quasinorms, by the following Radon--Nikodym-type domination principle.
For the proof in the case $X = \CC$, which extends to general Banach spaces, see \cite[Lemma 2.2]{gU16} and \cite[Proposition 3.6]{DT15}

\begin{prop}[Radon--Nikodym-type domination]\label{prop:outer-RN}
  Let $\XX$ be a Polish space, and let $(\EE,\sigma,X,S)$ be an outer structure on $\XX$ such that $\XX = \bigcup_{i \in \NN} E_i$ for some countable sequence of generating sets $E_i \in \EE$.
  If $\mathfrak{m}$ is a positive Borel measure on $\XX$ such that
  \begin{equation*}
    \int_E \| F(x) \|_{X} \, \dd \mf{m} (x) \lesssim \| F \|_{S(E)} \sigma(E) \qquad \forall E \in \EE, \, \forall F \in \Bor(\XX;X)
  \end{equation*}
  and
  \begin{equation*}
    \sigma(A) = 0 \, \Rightarrow \, \mf{m}(A) = 0 \qquad  \forall A \in \Bor(\XX), 
  \end{equation*}
  then 
  \begin{equation*}
     \int_\XX \| F(x) \|_{X} \, \dd \mf{m}(x)  \lesssim \|F\|_{L^1_\sigma S} \qquad \forall F \in \Bor(\XX;X).
  \end{equation*}
\end{prop}

The outer-$L^p$ spaces support a useful Marcinkiewicz-type interpolation theorem, proven in \cite[Proposition 3.5]{DT15} (see also \cite[Propostion 7.4]{DPO18}).
In applications we only prove bounds for outer-$L^p$ quasinorms by establishing endpoint weak outer-$L^p$ bounds.

\begin{prop}[Marcinkiewicz interpolation]\label{prop:outer-interpolation}
  Let $\XX$ be a Polish space, and let $(\EE, \sigma, X, S)$ be an outer structure on $\XX$.
  Let $\Omega$ be a $\sigma$-finite measure space, and let $T$ be a quasi-sublinear operator mapping $L^{p_1}(\Omega;X) + L^{p_2}(\Omega;X)$ into $\Bor(\XX;X)$ for some $1 \leq p_1 < p_2 \leq \infty$.
  Suppose that
  \begin{equation*}
    \begin{aligned}
      \|Tf\|_{L^{p_1,\infty}_\sigma S} &\lesssim \|f\|_{L^{p_1}(\Omega;X)}, \\
      \|Tf\|_{L^{p_2,\infty}_\sigma S} &\lesssim \|f\|_{L^{p_2}(\Omega;X)}
    \end{aligned}
    \qquad \forall f \in L^{p_1}(\Omega;X) + L^{p_2}(\Omega;X).
  \end{equation*}
  Then for all $p \in (p_1,p_2)$.
  \begin{equation*}
    \|Tf\|_{L^p_\sigma S} \lesssim \|f\|_{L^p(\Omega;X)} \qquad \forall f \in L^p(\Omega;X).
  \end{equation*}
  
\end{prop}

\subsection{Particular outer structures}\label{sec:particular-outer}

Now we move from general outer structures to those with relevance to Walsh time-frequency analysis.
Two collections of generating sets will be used: the collection $\TT$ of trees, and the collection $\DD$ of strips.
We will use the premeasures $\map{\mu}{\TT}{[0,\infty)}$ and $\map{\nu}{\DD}{[0,\infty)}$ defined by
\begin{equation*}
  \mu(T) := |I_T|, \qquad \nu(D) := |I_D|.
\end{equation*} 
Two families of sizes on $(\tPP,\TT)$, called `deterministic' and `randomised', will be needed.
The deterministic sizes are $\CC$-sizes, while the randomised sizes are $X^3$-sizes, where $X$ is a given Banach space.

\begin{defn}[Deterministic sizes]\label{def:deterministic-size}
  The $\CC$-sizes $S^1$ and $S^\infty$ on $(\tPP,\TT)$ are given by
  \begin{align*}
      \|F\|_{S^1(T)} &:=  \frac{1}{|I_{T}|} \sum_{\mb{P} \in T}| F(\mb{P}) |\, |I_{\mb{P}}| 
      \\
      \|F\|_{S^\infty(T)} &:= \sup_{\mb{P} \in T}| F(\mb{P}) |.
  \end{align*}
  for all $F \in \Bor(\tPP;\CC)$.\footnote{Note that every function on $\tPP$ is Borel, since $\tPP$ is countable.}
  We also define the mixed deterministic $\CC$-size $S^{(\infty,1)}$ by
  \begin{align*}
      \| F \|_{S^{(\infty,1)}(T)} &:=  \Bigl\|  \sum_{v=0}^{2}\sum_{\mb{P} \in T^{v}} |F(\mb{P})|\1_{I_{\mb{P}}}(x)\Bigr\|_{L^{\infty}}
    \\
    &
    =
    \sup_{x\in I_{T}}\sum_{\substack{\mb{P}\in T \\ I_{\mb{P}}\ni x}}
    | F (\mb{P})|.
  \end{align*} 

\end{defn}

\begin{defn}[Randomised sizes]\label{def:randomized-size}
  Let $X$ be a Banach space.
  The $X^{3}$-size $\RS$ is given for all $F \in \Bor(\tPP;X^3)$ by
  \begin{equation*}
    \|F\|_{\RS(T)} := \big\| \| F \|_{X^{3}} \big\|_{S^{\infty}(T)}+ \big\| \|\DEF F\|_{X^3} \big\|_{S^{(\infty,1)}(T)} + \sum_{u\in\{0,1,2\}} \|F\|_{\RS_{u}(T)},
  \end{equation*}
  where 
  \begin{equation}\label{eq:randomized-size}
    \|F\|_{\RS_{u}(T)} := 
      \sum\limits_{v\neq u}\bigg(\frac{1}{|I_{T}|} \int_{I_T} \E \bigg\| \sum_{\mb{P} \in T^{u}} \varepsilon_{\mb{P}} F(\mb{P}_{v}) \1_{I_{\mb{P}}}(x) \bigg\|_{X}^2 \, \dd x \bigg)^{1/2}.
  \end{equation}
\end{defn}

\begin{rmk}
  We do not mention the Banach space $X$ in the notation for the randomised size $\RS$; this should always be clear from context.
  Often we will refer to three functions $F_u \in \Bor(\tPP;X_u)$ ($u \in \{0,1,2\}$) valued in different Banach spaces, and discuss the three sizes $\|F_u\|_{\RS}$; here we have three different $X_u^3$-sizes $\RS$, but we gain no clarity from denoting these sizes differently.
\end{rmk}

It is almost clear that $\RS$ satisfies all the conditions of a size; the only subtlety is in showing that the component measuring $\DEF F$ satisfies the unconditionality property.

\begin{prop}
  Let $X$ be a Banach space, $F \in \Bor(\tPP;X)$, and suppose that $A \in \TT^{\cup}$ is a countable union of trees.
  Then
  \begin{equation}\label{eqn:defect-forest}
    \big\|  \| \DEF(\1_{A} F) \|_{X^{3}} \big\|_{S^{(\infty,1)}(T)}\lesssim \big\|  \| F \|_{X^{3}} \big\|_{S^{\infty}(T)}+ \big\| \| \DEF F \|_{X^{3}} \big\|_{S^{(\infty,1)}(T)}.
  \end{equation}
  It follows that $\RS$ satisfies the unconditionality property.
\end{prop}

\begin{proof}
  Notice that for all tritiles $\mb{P}$,
  \begin{equation}\label{eqn:defect-norm-est}
    \| \DEF(\1_{A} F)(\mb{P}) \|_{X^{3}} \leq      \| \DEF(F)(\mb{P}) \|_{X^{3}} +
\hspace{-1.5em}\sum_{\substack{\mb{Q}\leq \mb{P}\\ |I_{\mb{Q}}|=|I_{\mb{Q}}|/3}}\hspace{-1.5em}\Bigl( \| F(\mb{P}) \|_{X^{3}}+ \| F(\mb{Q})\|_{X^{3}}   \Bigr)\Bigl| \1_{A}(\mb{P}) - \1_{A}(\mb{Q})\Bigr|
  \end{equation}
  by \eqref{eq:defect-norm-bound} and since $\DEF(\1_{A} F)(\mb{P})=\DEF F(\mb{P})$ unless $\mb{P}\notin A$ and $\mb{Q}\in A$, where $\mb{Q}$ is a tritile with $\mb{Q}\leq \mb{P}$ and $|I_{\mb{Q}}|=|I_{\mb{P}}|/3$. Since $A\in\TT^{\cup}$ , for any $x\in I_{T}$ it holds that there is at most one $\mb{P}\in\tPP$ such that
  \begin{equation*}
    \sum_{\substack{\mb{Q}\leq \mb{P}\\ |I_{\mb{Q}}|=|I_{\mb{Q}}|/3}} \Bigl| \1_{A}(\mb{P}) - \1_{A}(\mb{Q})\Bigr|  \neq 0 \text{ and } x\in I_{\mb{P}};
  \end{equation*}
  writing out the definition of $S^{(\infty,1)}$, one sees that this gives the required estimate.
\end{proof}

We make use of the two premeasures $\mu$ and $\nu$ on $\tPP$ by iterating this construction to obtain `iterated' outer structures.

\begin{defn}[Iterated outer structures]\label{def:local-lp-size}
  Let $X$ be a Banach space.  Given an $X$-size $S$ on $(\tPP,\TT)$, for all $q \in (0,\infty)$ we define an $X$-size $\sL^q_\mu S$ on $(\tPP,\DD)$ by
  \begin{equation}
    \begin{aligned}
       \| F \|_{\sL^q_\mu S(D)} := |I_{D}|^{-1/q}\|\1_D F \|_{L^q_\mu S} \qquad \forall D \in \DD.
    \end{aligned}
  \end{equation}
  It is straightforward to verify that this is indeed an $X$-size on $(\tPP,\DD)$, and thus $(\DD,\nu,X,\sL^q_\mu S)$ is an \emph{iterated outer structure} on $\tPP$, inducing \emph{iterated outer-$L^p$ quasinorms} $\|\cdot\|_{L^p_\nu \sL^q_\mu S}$ for all $p \in (0,\infty]$.
\end{defn}

The following iterated outer H\"older inequality is a straightforward consequence of the `non-iterated' outer H\"older inequality of Proposition \ref{prop:outer-holder}. 

\begin{cor}[Hölder inequality for iterated outer-$L^{p}$ spaces]\label{cor:iterated-holder}
  Let $X_0, X_1, X_2,X$ be Banach spaces, and let $\map{\Pi'}{X_0\times X_1 \times X_2}{X}$ be a bounded trilinear form.\footnote{Here we write $\Pi'$ to emphasise that this is not the trilinear form that we consider in the introduction, and throughout most of the paper; in practise $\Pi'$ will be an extension of the aforementioned $\Pi$.}  Let $S$ be a $X$-size on $(\tPP,\TT)$, and for each $u \in \{0,1,2\}$, let $S_u$ be an $X_u$-size on $(\tPP,\TT)$ such that the size-H\"older inequality
  \begin{equation*}
    \| \Pi'(F_0,F_1,F_2) \|_{S} \lesssim \prod_{u=0}^2 \|  F_u \|_{S_{u}} \qquad \forall F_u \in \Bor(\tPP;X_u^{3})
  \end{equation*}
  holds. Then for all $p_u, q_u \in [1,\infty]$,
  \begin{equation*}
    \| \Pi'(F_{0},F_{1},F_{2}) \|_{L_\nu^p \sL_\mu^q S} \lesssim_{p_u,q_u} \prod_{u=0}^{2} \| F_{u} \|_{L_\nu^{p_u} \sL_\mu^{q_u} S_{u}} \qquad \forall F_u \in \Bor(\tPP;X_u^{3})
  \end{equation*}
  where $p^{-1} = \sum_{u=0}^2 p_u^{-1}$ and $q^{-1} = \sum_{u=0}^2 q_u^{-1}$.
\end{cor}
%

We return to consideration of our trilinear form $\map{\Pi}{X_0\times X_1 \times X_{2}}{\CC}$.
Define an `extended' trilinear form $\Pi^{*}\colon X_{0}^{3}\times X_{1}^{3}\times X_{2}^{3}\to \CC$ by
  \begin{equation*}
    \Pi^{*}\Bigl( (x_{0,0},x_{0,1},x_{0,2}),(x_{1,0},x_{1,1},x_{1,2}),(x_{2,0},x_{2,1},x_{2,2}) \Bigr):= \Pi(x_{0,0},x_{1,1},x_{2,2}).
  \end{equation*}

The most important result of this section is the following size-H\"older inequality for the randomised sizes and the deterministic size $S^1$.

\begin{prop}[Size-H\"older]\label{prop:new-size-holder}
  Let $X_0,X_1,X_2$, $\Pi$, and $\Pi^*$ be as above.
  Then
  \begin{equation*}
    \| \Pi^{*}(F_0,F_1, F_2) \|_{S^1(T)} \leq \prod_{u=0}^{2} \|F_u\|_{\RS(T)} \qquad \forall T \in \TT, \, \forall F_u \in \Bor(\tPP;X_u^{3}).
  \end{equation*}
\end{prop}

\begin{proof}
  First note that
  \begin{equation*}
    \begin{aligned}
      \sum_{\mb{P} \in T} \Bigl|\Pi^{*}\bigl(F_0(\mb{P}), F_1(\mb{P}), F_2(\mb{P})\bigr)\Bigr|\,|I_{\mb{P}}|
    &=
    \sum_{\mb{P} \in T} \Bigl|\Pi\bigl(F_0(\mb{P}_{0}), F_1(\mb{P}_{1}), F_2(\mb{P}_{2})\bigr)\Bigr|\,|I_{\mb{P}}|
    \\
    &
    \leq \sum_{u=0}^2 \sum_{\mb{P} \in T^u} \Bigl|\Pi \bigl(F_0(\mb{P}_{0}), F_1(\mb{P}_{1}), F_2(\mb{P}_{2})\bigr)\Bigr|\, |I_{\mb{P}}|,
  \end{aligned}
  \end{equation*}
  so it suffices to fix $u \in \{0,1,2\}$ and deal with the summands in the last entry individually. We concentrate on the case $u=0$; the other cases are analogous. We restrict the sum over $\mb{P}$ to
  \begin{equation*}
    \tPP_{N}=\Bigl\{\mb{P}\in\tPP \colon   |I_{\mb{P}}|>2^{-N} \Bigr\}
  \end{equation*}
  and we look for a bound independent of $N$, allowing us to conclude by standard limiting arguments.  For ease of notation we set $F_{0}^{N}(\mb{P}):=\1_{\tPP^{N}}(\mb{P})F_{0}(\mb{P})$.

  Fix a normalised sequence $a \in \ell^\infty(T^u;\CC)$ and estimate by duality
  \begin{equation*}
    \begin{aligned}
      &
       \sum_{\mb{P} \in T^0} \Bigl|\Pi\bigl( F_0^{N}(\mb{P}_{0}), F_1(\mb{P}_{1}), F_2(\mb{P}_{2})\bigr)\Bigr|\, |I_{\mb{P}}| 
      \\
      &
      = \sum_{\mb{P} \in T^0} a_{\mb{P}} \Pi\bigl(F_0^{N}(\mb{P}_{0}), F_1(\mb{P}_{1}), F_2(\mb{P}_{2})\bigr)\, |I_{\mb{P}}| 
      \\
        &\leq  \int_{I_T} \sum_{\mb{P} \in T^0}   a_{\mb{P}} \Pi\bigl(   \DEF F_0^{N}(\mb{P}_{0}) , F_1(\mb{P}_{1}), F_2(\mb{P}_{2})\bigr) \1_{I_\mb{P}}(x) \, \dd x \\ 
        & +  \sum_{\mb{P} \in T^0}  a_{\mb{P}} \Pi\Bigl(  \Bigl\langle\sum_{v=0}^{2} \sum_{\substack{\mb{Q}\in T^{v} \\ |I_{\mb{P}}|>|I_{\mb{Q}}|}} \DEF F_{0}^{N}(\mb{Q}_{v}) w_{\mb{Q}_{v}} |I_{\mb{Q}}| ;w_{\mb{P}_{0}}\Bigr\rangle , F_1(\mb{P}_{1}) , F_2(\mb{P}_{2})\Bigr)  \, |I_{\mb{P}}|.
      \end{aligned}
 \end{equation*}
 We bound the first summand as follows:
 \begin{equation*}
   \begin{aligned}
     &\int_{I_T} \sum_{\mb{P} \in T^0}  a_{\mb{P}} \Pi\bigl(   \DEF F_0(\mb{P}_{0}) , F_1(\mb{P}_{1}), F_2(\mb{P}_{2})\bigr) \1_{I_\mb{P}}(x) \, \dd x
     \\
     &\lesssim  |I_{T}|\; \bigl\| \| \DEF F_{0}  \|_{X_{0}^{3}}\bigr\|_{S^{(\infty,1)}}\;\bigl\|  \| F_{1} \|_{X_{1}^{3}}\bigr\|_{S^{\infty} }\; \bigl\|  \| F_{2} \|_{X_{2}^{3}}\bigr\|_{S^{\infty}}.
   \end{aligned}
 \end{equation*}
 As for the second summand,
 \begin{equation*}
   \begin{aligned}
     &  \sum_{\mb{P} \in T^0}
     a_{\mb{P}}\Pi\Bigl(
       \Bigl\langle
         \sum_{v=0}^{2} \sum_{\substack{\mb{Q}\in T^{v} \\ |I_{\mb{P}}|>|I_{\mb{Q}}|}}
         \DEF F_{0}^{N}(\mb{Q}_{v}) w_{\mb{Q}_{v}} |I_{\mb{Q}}| ; w_{\mb{P}_{0}}
       \Bigr\rangle
       , F_1(\mb{P}_{1}) , F_2(\mb{P}_{2})
     \Bigr)\, |I_{\mb{P}}|
   \\
   &
   =
 \sum_{\mb{P} \in T^0} a_{\mb{P}} \Bigl\langle\sum_{v=0}^{2} \sum_{\substack{\mb{Q}\in T^{v} \\ |I_{\mb{P}}|>|I_{\mb{Q}}|}}   \Pi\Bigl(  \DEF F_{0}^{N}(\mb{Q}_{v})  , F_1(\mb{P}_{1}) , F_2(\mb{P}_{2})\Bigr)  w_{\mb{Q}_{v}} |I_{\mb{Q}}| ;w_{\mb{P}_{0}}\Bigr\rangle \, |I_{\mb{P}}|
   \\
   &
   =
\sum_{v=0}^{2} \sum_{\mb{Q} \in T^v} \sum_{\substack{\mb{P}\in T^{0} \\ I_{\mb{P}}\supsetneq I_{\mb{Q}}}}   \Pi\Bigl(  \DEF F_{0}^{N}(\mb{Q}_{v})  , F_1(\mb{P}_{1}) , F_2(\mb{P}_{2})\Bigr)  a_{\mb{P}}\, b_{\mb{P},\mb{Q}} \, |I_{\mb{Q}}|
\end{aligned}
 \end{equation*}
 where the coefficients $b_{\mb{P},\mb{Q}} := \langle w_{\mb{Q}_{v}}; w_{\mb{P}_{0}}\rangle\, |I_{\mb{P}}|$ satisfy $|b_{\mb{P},\mb{Q}}|<1$. Letting $\varepsilon_{\mb{P}}$ be independent Rademacher variables, we have
 \begin{equation*}
   \begin{aligned}
     &
     \sum_{v=0}^{2} \sum_{\mb{Q} \in T^v} \sum_{\substack{\mb{P}\in T^{0} \\ I_{\mb{P}}\supsetneq I_{\mb{Q}}}}   \Pi\Bigl(  \DEF F_{0}^{N}(\mb{Q}_{v})  , F_1(\mb{P}_{1}) , F_2(\mb{P}_{2})\Bigr)  a_{\mb{P}}\, b_{\mb{P},\mb{Q}} \, |I_{\mb{Q}}|
\\
   &=
 \sum_{v=0}^{2} \sum_{\mb{Q} \in T^v}\int_{I_{T}} \sum_{\substack{\mb{P}\in T^{0} \\ I_{\mb{P}}\supsetneq I_{\mb{Q}}}}   \Pi\Bigl(  \DEF F_{0}^{N}(\mb{Q}_{v})  ,  a_{\mb{P}}\, b_{\mb{P},\mb{Q}} F_1(\mb{P}_{1})  \1_{I_{\mb{P}}}(x), F_2(\mb{P}_{2}) \1_{I_{\mb{P}}}(x)\Bigr) \1_{I_{\mb{Q}}}(x) \, \dd x
   \\
   &
   \begin{aligned}
     = \sum_{v=0}^{2} \sum_{\mb{Q} \in T^v}\int_{I_{T}} \E\; \Pi\biggl(  \DEF F_{0}^{N}(\mb{Q}_{v})  ,  &\sum_{\substack{\mb{P}\in T^{0} \\ I_{\mb{P}}\supsetneq I_{\mb{Q}}}}  \varepsilon_{\mb{P}} a_{\mb{P}}\, b_{\mb{P},\mb{Q}} F_1(\mb{P}_{1})  \1_{I_{\mb{P}}}(x),
     \\
     &\sum_{\substack{\mb{P}'\in T^{0} \\ I_{\mb{P}'}\supsetneq I_{\mb{Q}}}} \varepsilon_{\mb{P}'}  F_2(\mb{P}'_{2}) \1_{I_{\mb{P}'}}(x)\biggr) \,\1_{I_{\mb{Q}}}(x) \, \dd x
   \end{aligned}
   \\
   &
   \begin{aligned}
\lesssim\sum_{v=0}^{2} \sum_{\mb{Q} \in T^v} \int_{I_{T}}&\|  \DEF F_{0}^{N}(\mb{Q}_{v})  \|_{X_{0}}
   \prod_{u=1,2} \Bigl( \E \| \sum_{\substack{\mb{P}\in T^{0} \\ I_{\mb{P}}\supsetneq I_{\mb{Q}}}}  \varepsilon_{\mb{P}}  F_u(\mb{P}_{u})  \1_{I_{\mb{P}}}(x) \|_{X_{u}}^{2} \Bigr)^{1/2} \1_{I_{\mb{Q}}}(x) \, \dd x
   \\
 \end{aligned}
   \\
   &
   \begin{aligned}
     \lesssim\sum_{v=0}^{2}  \int_{I_{T}}&\Bigl(  \sum_{\mb{Q} \in T^v} \|  \DEF F_{0}^{N}(\mb{Q}_{v})  \|_{X_{0}}\1_{I_{\mb{Q}}}(x)\Bigr) \prod_{u=1}^2 \sup_{\mb{Q}\in T^{v}}\Bigl( \E \Bigl\| \sum_{\substack{\mb{P}\in T^{0} \\ I_{\mb{P}}\supsetneq I_{\mb{Q}}}}  \varepsilon_{\mb{P}} F_u(\mb{P}_{u})  \1_{I_{\mb{P}}}(x) \Bigr\|_{X_{u}}^{2} \Bigr)^{1/2} \, \dd x.
 \end{aligned}
   \end{aligned}
 \end{equation*}
 Finally, applying Cauchy--Schwartz to the last entry we obtain
\begin{equation*}
   \begin{aligned}
     &  \sum_{\mb{P} \in T^0}
     a_{\mb{P}}\Pi\Bigl(
       \Bigl\langle
         \sum_{v=0}^{2} \sum_{\substack{\mb{Q}\in T^{v} \\ |I_{\mb{P}}|>|I_{\mb{Q}}|}}
         \DEF F_{0}^{N}(\mb{Q}_{v}) w_{\mb{Q}_{v}} |I_{\mb{Q}}| ; w_{\mb{P}_{0}}
       \Bigr\rangle
       , F_1(\mb{P}_{1}) , F_2(\mb{P}_{2})
     \Bigr)\, |I_{\mb{P}}|
   \\
   &
   \lesssim \big\| \| \DEF F_{0} \|_{X_{0}^{3}} \big\|_{S^{(\infty,1)}(T)} \| F_{1} \|_{\RS(T)} \| F_{2} \|_{\RS(T)}
 \end{aligned}
\end{equation*}
concluding the proof.
\end{proof}

This H\"older inequality, combined with Radon--Nikodym domination, leads to the following result.

\begin{cor}\label{cor:rand-holder}
  Let $X_0,X_1,X_2$, $\Pi$, and $\Pi^*$ be as above.
  Let $(p_0,p_1,p_2)$ and $(q_0,q_1,q_2)$ be Hölder triples of exponents.
  Then
  \begin{equation*}
    \sum_{\mb{P} \in \tPP} |\Pi^*(F_0(\mb{P}),F_1(\mb{P}),F_2(\mb{P}))||I_{\mb{P}}| \lesssim \prod_{v=0}^2 \| F_v \|_{L^{p_v}_\mu \RS}
  \end{equation*}
  and
  \begin{equation}\label{eqn:iterated-RN-dom}
    \sum_{\mb{P} \in \tPP} |\Pi^*(F_0(\mb{P}),F_1(\mb{P}),F_2(\mb{P}))||I_{\mb{P}}| \lesssim \prod_{v=0}^2 \| F_v \|_{L^{p_v}_\nu \sL^{q_v}_\mu \RS}
  \end{equation}
  for all $F_v \in \Bor(\tPP;X_v^3)$.
\end{cor}

\begin{proof} 
  The first estimate follows from combining the outer H\"older inequality (Proposition \ref{prop:outer-holder}) with Radon--Nikodym domination (Proposition \ref{prop:outer-RN}), using Proposition \ref{prop:new-size-holder}.
  For the second, we have
  \begin{equation*}
    \| \Pi^*(F_0,F_1,F_2) \|_{L^1_\mu S^1} \lesssim \prod_{v=0}^2 \| F_v \|_{L^{q_v}_\mu \RS_{v}}
  \end{equation*}
  as a consequence of the outer Hölder inequality.
  By multiplying by characteristic functions of strips this implies
  \begin{equation}\label{eqn:iterated-size-holder}
    \| \Pi^*(F_0, F_1, F_2) \|_{\sL^1_\mu S^1} \lesssim \prod_{v=0}^2 \| F_v \|_{\sL^{q_v}_\mu \RS_{v}}.
  \end{equation}
  For each $\map{F}{\tPP}{\CC}$ and each strip $D \in \DD$, Radon--Nikodym domination yields
  \begin{equation*}
    \sum_{\mb{P} \in \tPP} |\1_D F(\mb{P})|\,|I_{\mb{P}}|  \leq \| \1_{\mc{D}} F\|_{L^1_\mu S^1} = \nu(D) \|F \|_{\sL^1_\mu S^1(D)}, 
  \end{equation*}
  so that
  \begin{equation*}
    \frac{1}{\nu(D)} \sum_{\mb{P} \in D} |F(\mb{P})|\,|I_{\mb{P}}|  \leq \|f\|_{\sL^1_\mu S^1(D)}.
  \end{equation*}
  Applying Radon--Nikodym domination and the iterated outer H\"older inequality (Corollary \ref{cor:iterated-holder}) completes the proof.
\end{proof}

\begin{rmk}\label{rmk:tritile-reduction}
  The tritile form associated to $\Pi \colon X_0 \times X_1 \times X_2 \to \CC$ can be written as
 \begin{equation*}
   \Lambda_\Pi(f_0,f_1,f_2) = \sum_{\mb{P} \in \tPP} \Pi^*\Big(\mc{E}[f_0](\mb{P}), \mc{E}[f_1](\mb{P}), \mc{E}[f_2](\mb{P})\Big),
 \end{equation*}
 so by Corollary \ref{cor:rand-holder} we have
 \begin{equation*}
   |\Lambda_\Pi(f_0,f_1,f_2)|
   \lesssim \prod_{u=0}^2 \|\mc{E}[f_u]\|_{L_\nu^{p_u} \sL_\mu^{q_u} \RS}.
 \end{equation*}
 Thus given a H\"older triple $(p_u)_{u=0}^2$, in order to prove the $L^p$-bounds
 \begin{equation*}
   |\Lambda_\Pi(f_0,f_1,f_2)| \lesssim \prod_{u=0}^2 \|f_u\|_{L^{p_u}(\WW;X_u)} \qquad \forall f_u \in \Sch(\WW;X_u),
 \end{equation*}
 it suffices to find a H\"older triple $(q_0,q_1,q_2)$ such that
 \begin{equation*}
   \|\mc{E}[f]\|_{L_\nu^{p_u} \sL_\mu^{q_u} \RS} \lesssim \|f\|_{L^{p_u}(\WW;X_u)} \qquad \forall f \in \Sch(\WW;X_u)
 \end{equation*}
 for all $u \in \{0,1,2\}$.
\end{rmk}

\subsection{Size domination}\label{sec:size-dom}

In this section we prove a `size domination' theorem, which allows us to control the randomised size $\RS$ of an embedded function $\Emb[f]$ by the deterministic size $S^\infty$.
This uses the UMD property of the Banach space under consideration; this property is not used anywhere else.
Thus in proving Theorem \ref{thm:intro-embeddings} we may replace the $\RS$ with $S^\infty$, which makes life easier.

\begin{thm}\label{thm:size-domination}
  Let $X$ be a UMD Banach space.
  Then for all convex $\AA\subset \tPP$,
  \begin{equation*}
    \bigl\|\1_{\AA} \mc{E}[f]\bigr\|_{\RS} \lesssim \bigl\|\1_{\AA} \| \mc{E}[f] \|_{X^{3}}\bigr\|_{S^\infty}, \qquad \forall f \in \Sch(\WW;X) 
  \end{equation*}
  with implicit constant independent of $\AA$.
\end{thm}

Convexity of a set of tritiles is defined in Definition \ref{defn:convex}.
Any of the standard norms on $X^3$ will do the job here, but we use the $\ell^\infty$-norm
\begin{equation*}
  \|  (x_0,x_1,x_2) \|_{X^{3}} = \sup_{u \in \{0,1,2\}} \|  x_u \|_{X}.
\end{equation*}
We will prove Theorem \ref{thm:size-domination} later in the section.
First we show how it implies outer-$L^{p}$ quasinorm bounds.
The argument is standard, but we include it to show the role played by convexity.

\begin{cor}\label{cor:outer-Lp-domination}
  Let $X$ be a UMD Banach space.
  Then for all convex $\AA\subset \tPP$, all $u \in \{0,1,2\}$, and all $p,q \in (1,\infty]$, we have the bounds
    \begin{equation} \label{eq:non-iterated-outer-Lp-domination}
    \bigl\|\1_{\AA}\,\Emb[f] \bigr\|_{L_\mu^p \RS} \lesssim \bigl\|\1_{\AA} \|\Emb[f] \|_{X^{3}} \bigr\|_{L_\mu^p S^{\infty}} \qquad \forall f \in \Sch(\WW;X)
  \end{equation}
 and
    \begin{equation} \label{eq:iterated-outer-Lp-domination}
    \bigl\|\1_{\AA}\,\Emb[f] \bigl\|_{L^{p}_{\nu}\sL_\mu^q \RS} \lesssim \bigl\|\1_{\AA} \|\Emb[f] \|_{X^{3}} \bigr\|_{L_{\nu}^{p} \sL_\mu^q S^{\infty}} \qquad \forall f \in \Sch(\WW;X).
  \end{equation}
\end{cor}

\begin{proof}
  Let us show that (\ref{eq:non-iterated-outer-Lp-domination}) holds.
  Assume that the right hand side of the inequality is finite. Then for each $n\in\Z$ there exists a countable union of trees $E_{n}=\bigcup_{i\in\N}T_{n,i}$ such that
  \begin{equation*}
    \sum_{n\in\Z} \mu(E_{n}) 2^{pn} \lesssim_{p} \|\1_{\AA} \|\Emb[f] \|_{X^{3}}\|_{L_\mu^{p} S^{\infty}}^{p} \quad \text{and} \quad \| \1_{\AA\setminus E_{n}} \| \Emb[f] \|_{X^{3}} \|_{S^\infty} \leq 2^{n}.
  \end{equation*}
  For each $n$ and $i$ the set $\tPP \setminus T_{n,i} $ is convex, and thus so is $\tPP \setminus E_{n} = \bigcap_{i\in\N}\bigl( \tPP \setminus T_{n,i}\bigr)$.
  Theorem \ref{thm:size-domination} implies that
  \begin{equation*}
     \| \1_{\AA\setminus E_{n} } \Emb[f] \|_{\RS} \lesssim 2^{n},
  \end{equation*}
  so by the definition of the outer-$L^{p}$ quasinorms it holds that
  \begin{equation*}
    \|\1_{\AA\setminus E_{n} } \Emb[f] \|_{L^{p}_{\mu}\RS}^{p} \lesssim \sum_{n\in\Z} \mu(E_{n}) 2^{pn}
  \end{equation*}
  as required.
  Similar reasoning yields the iterated bounds \eqref{eq:iterated-outer-Lp-domination}; it suffices to recall that strips and their complements are convex. 
\end{proof}

The proof of Theorem \ref{thm:size-domination} relies on the following lemma.
\begin{lem}\label{lem:convex-proj}
  Let $X$ be a Banach space and $f \in \Sch(\WW;X)$.  Let $T$ be a tree and $\AA$ a finite convex  set.  Then there exists a function $g \in \Sch(\WW;X)$ supported on $I_T$ such that
  \begin{equation}\label{eq:g-coeff}
    \Emb[g](\mb{P})=\Emb[f](\mb{P}) \qquad \forall \mb{P} \in T\cap\AA
  \end{equation}
  and
  \begin{equation}\label{eq:g-bound}
    \|g\|_{L^{\infty}(I_{T};X)} \lesssim \| \1_{\AA}  \| \Emb[f] \|_{X^{3}} \|_{S^\infty(T)}.
  \end{equation}
\end{lem}

\begin{proof}
  The set $\AA\cap T$ can be assumed to be non-empty, otherwise we can take $g=0$.
  We first reason under the assumption that $\mb{P}_{T}\in \AA$.
  Let
  \begin{equation*}
    \mc{J} := \bigcup_{\mb{P} \in \AA} \ch(I_{\mb{P}}),
  \end{equation*}
  where we recall that $\ch(J)$ denotes the set of triadic children of the interval $J$. By convexity of $\AA$, the set $\mc{J}$ satisfies
  \begin{equation}\label{eq:J-convexity}
    J\in \mc{J},\, J\subset J'\subsetneq I_{T} \implies J'\in\mc{J}.
  \end{equation}
  Let $\bar{\mc{J}}$ be the partition of $I_{T}$ generated by $\mc{J}$, i.e. the elements of $\bar{\mc{J}}$ are the maximal triadic subintervals of $I_{T}$, ordered by inclusion, that do not contain any interval of $\mc{J}$ as a proper subset.
  The set $\bar{\mc{J}}$ can also be characterised as the set of minimal elements of $\mc{J}$ with respect to inclusion.
  It follows that for any $J\in\bar{\mc{J}}$ there exists a unique $\mb{P}(J)$ such that $J \in \ch(I_{\mb{P}(J)}) $.
  Furthermore, for any $\mb{P}\in\AA$, the elements of $\bar{\mc{J}}$ cannot contain $I_{\mb{P}}$, and thus $\{J\in\bar{\mc{{J}}}\colon J\subset I_{\mb{P}}\}$ partitions $I_{\mb{P}}$. 
  
 For every $J\in\bar{\mc{J}}$ let $Q_{J}$ be the unique tile such that $\xi_{T}\in\omega_{Q_{J}}$ and $I_{Q_{J}}=J$, and set 
  \begin{equation*}
    g:=\sum_{J\in\bar{\mc{J}}} \langle f; w_{Q_{J}}\rangle w_{Q_{J}} |J|. 
  \end{equation*}
  Let us show that \eqref{eq:g-coeff} holds.
  Given any $\mb{P} \in \AA$ the intervals  $\{J\in\bar{\mc{{J}}}\colon J\subset I_{\mb{P}}\}$ partition $I_{\mb{P}}$, and since any such $J$ does not properly contain any of the triadic children of $I_{\mb{P}}$, it holds that  $|J|\leq|I_{\mb{P}}|/3$ and thus $|\omega_{\mb{P}}|=3/|I_{\mb{P}}|\le |\omega_{Q_{J}}|$.
  Since $\xi_{T}\in \omega_{P_{J}}\cap \omega_{\mb{P}}$ this implies that
  \begin{equation*}
   \mb{P}\subset\bigcup_{J\in\bar{\mc{J}},\, J\subset I_{P}} Q_{J}.
  \end{equation*}
  By Lemma \ref{lem:wave-packet-interaction}, for any $u\in\{0,1,2\}$ it holds that
  \begin{equation*}
      w_{\mb{P}_{u}} = \sum_{J\in\bar{\mc{J}},\; J\subset I_{\mb{P}}} \langle w_{\mb{P}_{u}}; w_{Q_{J}}\rangle w_{Q_{J}}  |J|.
  \end{equation*}
  It follows that
    \begin{align*}
      \langle f ; w_{\mb{P}_{u}}\rangle
      &= \sum_{J\in\bar{\mc{J}},\; J\subset I_{\mb{P}}} \langle f ; w_{Q_{J}} (x)\rangle \langle w_{Q_{J}}; w_{\mb{P}_{u}}\rangle   |J| \\
      &= \langle g ; w_{\mb{P}_{u}}\rangle -  \sum_{J\in\bar{\mc{J}},\; J\not \subset I_{\mb{P}}} \langle f ; w_{Q_{J}} (x)\rangle \langle w_{Q_{J}}; w_{\mb{P}_{u}}\rangle   |J|
      \\
      &= \langle g ; w_{\mb{P}_{u}}\rangle ,
  \end{align*}
where the last equality holds by maximality of $\bar{\mc{J}}$: if $J\in\bar{\mc{J}}$ and $J\not \subset I_{\mb{P}}$ then $J\cap I_{\mb{P}}=\emptyset$.
  
Now we prove the bound \eqref{eq:g-bound}.
The wave packets $w_{Q_{J}}$ for $J\in\bar{\mc{J}}$ have disjoint time support so it suffices to show that
\begin{equation*}
  \|\langle f ; w_{Q_{J}}\rangle\|_{X} \lesssim \sup_{\substack{\mb{P}\in\AA \\ u \in \{0,1,2\}}} \bigl\|\Emb_{u}[f](\mb{P})\bigr\|_{X}
\end{equation*}
for all such $J$.
Notice that $Q_{J}\subset\bigcup_{u=0}^{2}\mb{P}(J)_{u}$  with $\mb{P}(J)$ as above, so using Lemma \ref{lem:wave-packet-interaction} we obtain that
\begin{align*}
  &\|\langle f ; w_{Q_{J}}\rangle\|_{X} \leq \sum_{u\in\{0,1,2\}} \|\langle f; w_{\mb{P}(J)_{u}}\rangle\|_{X} \, \bigl|\langle w_{Q_{J}}; w_{\mb{P}(J)_{u}}\rangle\bigr|\, |I_{\mb{P}(J)_{u}}|
  \\
  & \leq\sum_{u\in\{0,1,2\}} \|\langle f; w_{\mb{P}(J)_{u}}\rangle\|_{X} \lesssim \sup_{\substack{\mb{P}\in\AA \\ u \in \{0,1,2\}}}  \bigl\|\Emb_{u}[f](\mb{P})\bigr\|_{X}
\end{align*}
as required.

Finally, suppose that $\mb{P}_{T}\notin \AA$. Let $(\mb{O}_i)_i$ be the maximal elements of $T\cap\AA$ with respect to the order $\leq$.
The intervals $I_{\mb{O}_{i}}$ are pairwise disjoint, and $T\cap \AA$ can be written as a union of disjoint sets $\cup_{i}T(\mb{O}_{i})\cap\AA$.
Applying the above reasoning to each $T(\mb{O}_{i})$ we obtain a set of disjointly supported functions $g_{i}$ satisfying
\begin{align*}
  &\Emb[g_{i}](\mb{P})=\Emb[f](\mb{P}) \qquad \forall \mb{P}\in T(\mb{O}_{i})\cap \AA
  \\
  &
    \|g_{i}\|_{L^{\infty}(I_{T};X)} \lesssim \| \1_{\AA} \| \Emb[f] \|_{X^{3}} \|_{S^\infty(T(\mb{O}_{i}))}.
\end{align*}
Setting $g=\sum_{i}g_{i}$ completes the proof. 
\end{proof}

Recall that the randomised size $\RS$ is the sum of three types of terms,
\begin{equation*}
  \| \1_\AA \Emb[f] \|_{\RS(T)} = \big\| \1_\AA \|\Emb[f]\|_{X^3} \big\|_{S^\infty(T)}
  + \big\| \|\DEF (\1_\AA \Emb[f])\|_{X^3} \big\|_{S^{(1,\infty)}(T)}
  + \sum_{u \in \{0,1,2\}} \|\Emb[f]\|_{\RS_u(T)}.
\end{equation*}
The first summand need not be estimated; we handle the remaining summands separately.

\begin{prop}[Defect size domination]\label{prop:defect-size-domination}
  Let $X$ be a Banach space and $\AA\subset \tPP $ a convex set.
  Then for all trees $T$ and all $f \in \Sch(\WW;X)$,
  \begin{equation}\label{eq:deterministic-defect}
    \big\|  \| \DEF(\1_{\AA} \Emb[f]) \|_{X^{3}} \big\|_{S^{(\infty,1)}(T)}\lesssim \big\|  \1_{\AA} \| \Emb[f] \|_{X^{3}} \big\|_{S^{\infty}(T)}.
  \end{equation}
\end{prop}

\begin{proof}
  Using the estimate \eqref{eqn:defect-norm-est} and the fact that $\DEF \Emb[f] = 0$, for all tritiles $\mb{P}$ we have

 
  \begin{equation*}
    \| \DEF(\1_{A}\Emb[f])(\mb{P})\|_{X^{3}}\lesssim
    \sum_{\substack{\mb{Q}\leq \mb{P}\\ |I_{\mb{Q}}|=|I_{\mb{Q}}|/3}}\hspace{-1.5em}\Bigl( \|
    \1_{\AA}\Emb[f](\mb{P}) \|_{X^{3}}+ \|\1_{\AA}\Emb[f](\mb{Q})\|_{X^{3}}   \Bigr)\Bigl| \1_{\AA}(\mb{P}) - \1_{\AA}(\mb{Q})\Bigr|.
  \end{equation*}
  Since $\AA$ is convex, for each $x\in I_{T}$ there are at most two tritiles $\mb{P}$ such that $x\in I_{\mb{P}}$ and
  \begin{equation*}
\sum_{\substack{\mb{Q}\leq \mb{P}\\ |I_{\mb{Q}}|=|I_{\mb{Q}}|/3}}\hspace{-1.5em}\Bigl| \1_{\AA}(\mb{P}) - \1_{\AA}(\mb{Q})\Bigr|\neq0
\end{equation*}
It follows that 
\begin{align*}
    \big\|   \| \DEF(\1_{A}\Emb[f])(\mb{P})\|_{X^{3}} \big\|_{S^{(\infty,1)}}
    &\lesssim \big\|\| \1_{\AA} \Emb[f] \|_{X^{3}} \Bigr\|_{S^{\infty}(T)}
    \Bigl\|      \mb{P} \mapsto\hspace{-1.5em}\sum_{\substack{\mb{Q}\leq \mb{P}\\ |I_{\mb{Q}}|=|I_{\mb{Q}}|/3}}\hspace{-1.5em}\Bigl| \1_{\AA}(\mb{P}) - \1_{\AA}(\mb{Q})\Bigr|\Bigr\|_{S^{(\infty,1)}(T)} \\
    &\lesssim \big\|\| \1_{\AA} \Emb[f] \|_{X^{3}} \big\|_{S^{\infty}(T)}
\end{align*}
as required.
\end{proof}

\begin{prop}[Lacunary size domination]\label{prop:lacunary-size-domination}
  Let $X$ be a UMD Banach space and $\AA \subset \tPP$ a convex set.
  Then for all trees $T$, all $f \in \Sch(\WW;X)$, and all $u \in \{0,1,2\}$,
  \begin{equation}\label{eqn:lac-size-dom-est}
    \|\1_\AA \Emb[f]\|_{\RS_u(T)} \lesssim \big\| \1_\AA \|\Emb[f]\|_{X^3} \big\|_{S^\infty(T)}.
  \end{equation}
\end{prop}

\begin{proof}
  Let $\tPP_{N}=\{P \in \tPP : 3^{-N}<|I_{P}|<3^{N}\}$.
  We show that
  \begin{equation*}
     \| \1_{\AA\cap\tPP_{N}} \mc{E}[f] \|_{\RS_u(T)} \lesssim \sup_{\substack{\mb{P} \in \AA\cap T\cap \tPP_{N} \\ v \in \{0,1,2\}}} \| \langle f; w_{\mb{P}_v} \rangle\|_{X}
  \end{equation*}
  for fixed $N$; the theorem follows by passing to the limit $N \to \infty$.
  Since $\AA\cap T \cap\tPP_{N} $ is finite and convex, by Lemma \ref{lem:convex-proj} there exists a function $g \in \Sch(\WW;X)$ supported on $I_T$ such  that
  \begin{align*}
      &\|g\|_{L^{\infty}(I_{T};X)} \lesssim \sup_{\substack{\mb{P} \in  T \cap\AA \cap \tPP_{N} \\ v \in \{0,1,2\}}}
      \|\langle f ; w_{\mb{P}_v}\rangle\|_{X}\\
      & \langle g; w_{\mb{P}_{v}}\rangle = \langle f; w_{\mb{P}_{v}}\rangle \qquad \forall \mb{P} \in  T \cap \AA \cap \tPP_{N}, \;v\in\{0,1,2\}.
  \end{align*}
  Now fix $v \in \{0,1,2\} \sm \{u\}$.
  Since $X$ is UMD we have by Proposition \ref{prop:lacunary-tree-proj}
  \begin{align*}
   |I_T|^{-1/2} \E \biggl\|&\sum_{\mb{P} \in   T^u \cap \AA \cap \tPP_{N} } \varepsilon_{\mb{P}} \langle f; w_{\mb{P}_v} \rangle w_{\mb{P}_v}\,|I_{\mb{P}}| \biggr\|_{L^2(I_T;X)} \\
                          &\lesssim|I_T|^{-1/2} \E \biggl\|\sum_{\mb{P} \in \tPP_N \cap T^u} \varepsilon_{\mb{P}} \langle g; w_{\mb{P}_v} \rangle w_{\mb{P}_v}\,|I_{\mb{P}}| \biggr\|_{L^2(I_T;X)} \\
    &\lesssim |I_T|^{-1/2} \|g\|_{L^2(I_T;X)} \leq \|g\|_{L^\infty(I_T;X)},
  \end{align*}
  Summing this over $v \neq u$ and using the $L^\infty$-bound on $g$ yields \eqref{eqn:lac-size-dom-est}.
  
\end{proof}

Theorem \ref{thm:size-domination} follows immediately from Propositions \ref{prop:defect-size-domination} and \ref{prop:lacunary-size-domination}.


\section{Proofs of the embedding bounds}
\label{sec:energy-embedding}
In this section we prove Theorem \ref{thm:intro-embeddings}: modulation invariant Carleson embedding bounds into iterated and non-iterated outer-$L^p$ spaces.
Before getting to the proofs themselves, we isolate a tile selection algorithm that appears multiple times in the proofs.
Thanks to the size domination theorem (Theorem \ref{thm:size-domination}), we only need this simple tile selection procedure, rather than a more complicated tree selection procedure (as used for example in \cite{HLP13}).

\begin{prop}[Tile selection]\label{prop:tile-selection}
  Let $F\in\Bor(\tPP;\C)$. For any $\lambda>0$ there exists a (possibly empty) set $\BB_{\lambda}$ of pairwise disjoint tritiles such that, if we set $E_{\lambda} := \bigcup_{\mb{B} \in \BB_{\lambda}} T(\mb{B})$,
  \begin{itemize}
  \item
    for each $\mb{B} \in \BB_{\lambda}$, $\bigl|F(\mb{B})\bigr| > \lambda$,
  \item
    for all $\mb{P} \in \tPP \setminus E_{\lambda}$, $\bigl|F(\mb{P})\bigr| \leq \lambda$.
  \end{itemize}
\end{prop}
  
\begin{proof}
  Let $\MM_{\lambda} := \{\mb{P} \in \tPP : \bigl|F(\mb{P})\bigr| > \lambda\}$.  If $\MM_{\lambda}=\emptyset$ then just set $\BB_{\lambda}=\emptyset$. Otherwise let $\BB_{\lambda}\subset \MM_\lambda$ be the subset of tritiles in $\MM_\lambda$ that are maximal with respect to $\leq$.
  Then $\BB_\lambda$ satisfies the first required condition, and to see the second one simply notes that $\MM_{\lambda} \subset E_{\lambda}$.
  To see that $\BB_{\lambda}$ consists of pairwise disjoint tritiles, suppose that $\mb{P},\mb{Q} \in \BB_{\lambda}$ with $\mb{P} \cap \mb{Q} \neq \emptyset$.
  Then either $\mb{P}\leq \mb{Q}$ or $\mb{Q}\le \mb{P}$, and by maximality of $\mb{P}$ and $\mb{Q}$ in $\MM_\lambda$ we must have that $\mb{P} = \mb{Q}$.
\end{proof}

We are ready to prove our modulation invariant Carleson embedding bounds.
We prove these with respect to the deterministic size $S^\infty$, under an $r$-Hilbertian assumption; we will obtain Theorem \ref{thm:intro-embeddings} as a corollary of the size domination theorem. First we consider embeddings into non-iterated outer-$L^p$ spaces. These are easier to prove, but they only hold for $p > r$.

\begin{thm}\label{thm:non-iterated}
  Let $X$ be a Banach space which is $r$-Hilbertian for some $r \in [2,\infty)$.\footnote{Alternatively, one can suppose that $X$ satisfies the `Walsh tile-type' bounds \eqref{eqn:tile-orthogonality}.}
  Then the bounds
    \begin{align*}
          \big\| \|\mc{E}[f]\|_{X^3}  \big\|_{L_\mu^p S^\infty} &\lesssim \|f \|_{L^p(\WW;X)}\qquad \forall p\in(r,\infty],
      \\
        \big\| \|\mc{E}[f]\|_{X^3}  \big\|_{L_\mu^{r,\infty} S^\infty} &\lesssim \|f \|_{L^r(\WW;X)}
  \end{align*}
  hold for all $f \in \Sch(\WW;X)$.
\end{thm}

\begin{proof}
  By interpolation (i.e. by Proposition \ref{prop:outer-interpolation}) it suffices to establish weak endpoint bounds for $p=\infty$ and $p=r$.
  The $p = \infty$ endpoint follows immediately from the definition of $S^\infty$:
  \begin{equation*}
    \big\|\|\mc{E}[f]\|_{X^3}  \big\|_{L_\mu^\infty S^\infty} =\sup_{\substack{\mb{P}\in\tPP \\ u\in\{0,1,2\}}} \| \langle f;w_{\mb{P}_{u}}\rangle \|_{X}\leq \| f \|_{L^{\infty}(\WW;X)}.
  \end{equation*}
  For the weak outer-$L^{r}$ endpoint, we need to show that for every $\lambda > 0$ there exists a set $E_\lambda \subset \tPP$ such that
  \begin{equation}\label{eqn:nonit-set-cond}
    \mu(E_\lambda) \lesssim \lambda^{-r} \|f\|_{r}^{r}
    \quad \text{and} \quad
    \big\| \1_{\tPP \sm E_\lambda} \|\mc{E}[f]\|_{X^3} \big\|_{S^\infty} \lesssim \lambda.
  \end{equation}
  Apply the tile selection (Proposition \ref{prop:tile-selection}) at level $\lambda$ to the function  $F(\mb{P})=\| \Emb[f](\mb{P}) \|_{X^{3}}$ to get a disjoint collection of tritiles $\BB_{\lambda}$ such that
\begin{equation*}
    \big\| \1_{\tPP \sm E_\lambda} \|\mc{E}_v[f]\|_{X^3} \big\|_{S^\infty}
    = \sup_{\substack{\mb{P} \in T \cap \tPP \sm E_\lambda \\ u \in \{0,1,2\}}} \|\langle f; w_{\mb{P}_u}\rangle\|_{X}
    \leq \lambda 
  \end{equation*}
  with $E_\lambda := \bigcup_{\mb{B} \in \BB} T(\mb{B})$. It remains to show the bound on $\mu(E_{\lambda})$.
  
  For each $\mb{B} \in \BB_{\lambda}$ there exists a tile $P_{\mb{B}}\in\mb{B}$ of the tritile $\mb{B}$ such that $\|\langle f; w_{P_{\mb{B}}} \rangle\|_{X} > \lambda$. The tritiles $\mb{B}$ are pairwise disjoint and thus so are the tiles $P_{\mb{B}}$; therefore we have
  \begin{equation*}
    \mu(E_\lambda)
    \leq \sum_{\mb{B} \in \BB_{\lambda}} |I_{\mb{B}}|
    \leq \lambda^{-r} \sum_{\mb{B} \in \BB_{\lambda}} \|\langle f; w_{P_{\mb{B}}} \rangle\|_X^r |I_{P_{\mb{B}}}|
    \lesssim \lambda^{-r} \|f\|_r^r
  \end{equation*}
  where the last estimate follows from Proposition \ref{prop:tile-orthogonality} applied to all finite subsets of $\BB_{\lambda}$.
\end{proof}

Now we prove the embeddings into iterated outer-$L^p$ spaces, which hold for all $p > 1$, but which are much harder to prove.

\begin{thm}\label{theorem:iterated-embeddings}
  Let $X$ be a Banach space which is $r$-Hilbertian for some $r \in [2,\infty)$.\footnote{Again, one could alternatively suppose that $X$ satisfies the bounds \eqref{eqn:tile-orthogonality}.}
  Then for all $p \in (1,\infty)$ and $q \in (\min(p,r)^\prime(r-1),\infty]$ the bound 
  \begin{equation*}
    \big\| \|\mc{E}[f]\|_{X^3} \big\|_{L_\nu^p \sL_\mu^q S^\infty} \lesssim \|f\|_{L^p(\WW;X)}
  \end{equation*}
  holds for all $f \in \Sch(\WW;X)$.
\end{thm}

\begin{proof}
  Fix $p \in (1,\infty)$. We will establish various endpoints depending on the position of $p$ relative to $r$; interpolation will then yield the estimates that we claim.
  In all cases, we will first fix $\lambda > 0$ and utilise the set $K_\lambda \subset 3\PP$ defined (dependent on $p$) as follows: write
  \begin{equation*}
    \big\{x \in \WW : M_{\min(p,r)}(\|f\|_{X})(x) > \lambda\big\} = \bigcup_{n \in \NN} I_{n,\lambda}
  \end{equation*}
  as a disjoint union of (maximal) triadic intervals, and then define
  \begin{equation*}
    K_\lambda := \bigcup_{n \in \NN} D(I_{n,\lambda}),
  \end{equation*}
  where $D(I_{n,\lambda})$ is the strip generated by $I_{n,\lambda}$. Since the $\min(p,r)$-maximal function $M_{\min(p,r)}$ is of weak type $(p,p)$, we then have
  \begin{equation*}
    \nu(K_\lambda) \leq  \sum_{n \in \NN} |I_n| \lesssim  \lambda^{-p} \| f \|_{L^p(\WW;X)}.
  \end{equation*}
  In each case it remains to show for an appropriate exponent $q$ that 
  \begin{equation}\label{eqn:K-s-endpoint}
     \big\| \1_{\tPP \setminus K_\lambda} \|\mc{E}[f] \|_{X^3} \big\|_{\sL^{q,\infty}_{\mu}S^\infty} \lesssim \lambda
  \end{equation}
  for all $\lambda > 0$.
  
  \textbf{Endpoint 1: $p < \infty$, $q = \infty$.}
  Here we need to show that
  \begin{equation*}
    \big\| \1_{\tPP \setminus K_\lambda} \|\mc{E}[f] \|_{X^3} \big\|_{\sL^{\infty}_{\mu} S^\infty}
    = \big\| \1_{\tPP \setminus K_\lambda} \|\mc{E}[f] \|_{X^3} \big\|_{S^\infty} \lesssim \lambda.
  \end{equation*}
  This follows from the definition of $K_\lambda$:
  \begin{equation*}
    \big\| \1_{\tPP \sm K_\lambda} \|\mc{E}[f] \|_{X^3} \big\|_{S^\infty}
    = \sup_{\substack{\mb{P} \colon I_{\mb{P}} \not\subset \bigcup_{n\in\NN} I_{n,\lambda} \\ u \in \{0,1,2\}}} \|\langle f; w_{\mb{P}_u} \rangle\|_X 
    \leq \sup_{I \not\subset \bigcup_{n \in \NN} I_{n,\lambda}} \|f\|_{\sL^1(I;X)} \leq \lambda. 
  \end{equation*}

  \textbf{Endpoint 2: $p \geq r$, $q = r$.}
  We must show that for every strip $D \in \DD$ and every $\tau > 0$ there exists $E_\tau \subset \tPP$ such that
  \begin{equation}\label{eqn:energy-2-cond-2}
    \mu(E_\tau) \lesssim (\lambda/\tau)^{r}  |I_D|
    \quad \text{and} \quad
    \big| \1_{(D\setminus K_\lambda)\setminus E_\tau} \|\mc{E}[f] \|_{X^3}\big|_{S^\infty} \lesssim \tau.
  \end{equation}
  It suffices to assume that $\tau < \lambda$, for otherwise we can take $E_\tau = \emptyset$ and the result follows from Endpoint 1.

  Fix a strip $D$. We may assume $I_{D}\not\subset I_{n,\lambda}$ for all $n\in\N$, since otherwise $D\setminus K_{\lambda}=\emptyset$ and there is nothing to prove. It thus holds that
  \begin{equation}
    \label{eq:local-lq-strip-bound}
    \| f \|_{\sL^{r}(I_{D};X)}\leq     \| f \|_{\sL^{p}(I_{D};X)}\lesssim \lambda.
  \end{equation}
  For $\mb{P} \in D$ we have that $\langle f; w_{\mb{P}_v}\rangle= \langle f\1_{I_{D}}; w_{\mb{P}_v}\rangle$ for all $v\in\{0,1,2\}$.
  The non-iterated version of the embedding, i.e. Theorem \ref{thm:non-iterated}, then guarantees that 
  \begin{equation*}
     \big\|  \1_{D \sm K_\lambda} \|\mc{E}[f] \|_{X^3} \big\|_{\sL^{r,\infty} S^\infty}^{r} \lesssim \| f \1_{I_{D}} \|_{L^{r}} \lesssim \lambda^{r} |I_{D}|
  \end{equation*}
  and we are done.
  
  \textbf{`Endpoint' 3: $p < r$ and $q> p'(r-1)$.}
  We will show that for every strip $D \in \DD$ and every $\tau > 0$ there exists $E_\tau \subset \tPP$ such that
  \begin{equation}\label{eqn:energy-2-cond-2}
    \mu(E_\tau) \lesssim (\lambda/\tau)^{q}  |I_D|
    \quad \text{and} \quad
    \big\| \1_{(D \setminus K_{\lambda})\setminus E_\tau}  \|\mc{E}[f]\|_{X^3} \big\|_{S^\infty} \lesssim \tau.
  \end{equation}
  for any $q>p'(r-1)$. The result of Endpoint 1 allows us to consider only $q$ close to $p'(r-1)$ and extend the result to all $q$ by interpolation. Furthermore  it suffices to assume that $\tau < \lambda$, for otherwise we can take $E_\tau = \emptyset$ and the result follows from the $s=\infty$ bound.

  Fix a strip $D$.  As before we may assume $I_D \not\subset I_{n,\lambda}$ for all $n\in\NN$, so if $I_{n,\lambda}$ intersects $I_D$, we must have $I_{n,\lambda} \subsetneq I_D$. Henceforth we consider  only those indices $n\in\N$ for which $I_{n,\lambda} \subsetneq I_D$, and we drop $\lambda$ from the notation.  For each $k \in \NN$ let $(J_{n,k,m})_{m\in\N}$ denote the maximal subintervals of $I_{n}$ on which $M_p(\|f\|_{X}) > 2^k\lambda$.  

  Let us decompose $f$ by setting 
  \begin{equation*}
    \begin{aligned}
      f\1_{I_{D}} &=f_{-1}+ \sum_{k=0}^{\infty}f_{k}, \\
       f_{-1}&=f\1_{I_{D}\setminus\bigcup_{n\in\NN}I_{n}} \qquad 
       f_{k} := \sum_{n\in\NN}\sum_{m\in\NN} f \1_{\Delta J_{n,k,m}} \quad \forall k \in \N,
    \end{aligned}
  \end{equation*}
  with
  \begin{equation*}
    \Delta J_{n,k,m}=J_{n,k,m} \setminus\Bigl(  \bigcup_{m' \in\NN} \bigcup_{k' > k} J_{n,k',m'} \Bigr).
  \end{equation*}
  We have bounds
  \begin{equation}\label{eq:fk-bounds}
    \begin{aligned}
      & \| f_{k} \|_{L^{\infty}(\WW;X)}\leq 2^{k}\lambda \qquad\forall k\in\{-1\}\cup \NN,
      \\
      &  \sum_{m\in\NN} |\Delta J_{n,k,m}| \leq \sum_{m \in \N} |J_{n,k,m}| \lesssim 2^{-kp} |I_n| \qquad \forall n,k\in\NN'
    \end{aligned}
  \end{equation}
  the latter follows from the weak $L^{p}$ boundedness of $M_{p}$.
  
  This decomposition induces the following decomposition of the embedded function $\Emb[f\1_{I_D}]$:
  \begin{equation}\label{eqn:F-rep}
    \Emb[f\1_{I_D}] =  \sum_{k=-1}^\infty F_k, \qquad F_{k}:= \Emb[f_{k}] \quad \forall k\in \{-1\}\cup \NN.
  \end{equation}
  Now fix $\varepsilon > 0$, and for each $k \geq -1$ apply the tile selection of Proposition \ref{prop:tile-selection} to $F_k$ at level $2^{-\varepsilon k}\tau$, yielding sets $\BB_k$ and  $\tilde{E}_k := \sum_{\mb{B} \in \BB_k} T(\mb{B})$ of tritiles such that
  \begin{equation}\label{eqn:Fk-selection-est-1}
    \big\| \1_{(D \setminus K_\lambda)\setminus \tilde{E}_k} \|F_k\|_{X^3} \big\|_{S^\infty} \lesssim 2^{-\varepsilon k}\tau
  \end{equation}
  and
  \begin{equation}\label{eqn:Fk-selection-est-2}
    \begin{aligned}
      \mu(\tilde{E}_k) \leq \sum_{\mb{B} \in \tilde{E}_k} |I_{\mb{B}}| &\leq (2^{-\varepsilon k}\tau)^{-r} \|f_k\|_{L^{r}(\WW;X)}^r\\
      &\leq \tau^{-r} 2^{kr\epsilon} \| f_{k} \|_{L^{\infty}(\WW;X)}^{r} \bigl|\spt(f_{k})\bigr|\lesssim \Bigl(  \frac{\lambda}{\tau}\Bigr)^{r} 2^{k(r(1+\epsilon)-p)} |I_D|.
    \end{aligned}
  \end{equation}
  On the other hand for any $\mb{P}\in D\setminus K_{\lambda}$ one has that
  \begin{equation*}
    \|F_{k}(\mb{P})\|_{X^{3}} = \sup_{\substack{\mb{P}\in D\setminus K_{\lambda}\\u\in\{0,1,2\}}}\bigl|\langle f_{k};w_{\mb{P}_{u}} \rangle\bigr| \lesssim 2^{k(1-p)} \lambda \qquad \forall k\in\{-1\}\cup \NN.
  \end{equation*}
  For $k=-1$ this is a trivial consequence of (\ref{eq:fk-bounds}), while for $k\in\NN$ notice that $I_{n}\cap I_{\mb{P}}\ne \emptyset$ only if $I_{n}\subset I_{\mb{P}}$ so
 \begin{equation*}
    \begin{aligned}
      \bigl|\langle f_{k}; w_{\mb{P}_{u}}\rangle\bigr| &\leq\| f_{k} \|_{L^{\infty}}  \sum_{n\colon I_{n}\subset I_{\mb{P}}} \frac{\bigl|\spt{f_{k}}\cap I_{n}\bigr|}{|I_{\mb{P}}|}\lesssim \| f_{k} \|_{L^{\infty}(\WW;X)}  \sum_{n\colon I_{n}\subset I_{\mb{P}}}\sum_{m\in\N} \frac{|\Delta J_{n,m,k}|}{|I_{\mb{P}}|}
      \\
      & \lesssim 2^{k}\lambda \sum_{n\colon I_{n}\subset I_{\mb{P}}} \frac{2^{-kp}|I_{n}|}{|I_{\mb{P}}|} \leq 2^{k(1-p)}\lambda. 
    \end{aligned}
  \end{equation*}
  It follows that $\tilde{E}_k$ is empty when $2^{k(p-1-\epsilon)} \gtrsim  \frac{\lambda}{\tau}$, i.e. when $k \geq \bar{k}_{\lambda/\tau}$ with $2^{\bar{k}_{\lambda/\tau}}\simeq (\lambda/\tau)^{\frac{1}{p-1-\epsilon}}$

  We conclude by setting $E_\tau := \bigcup_{k=-1}^{\bar{k}_{\lambda/\tau}} \tilde{E}_{k}$. Since $r(1+\epsilon)-p>0$, estimate \eqref{eqn:Fk-selection-est-2} gives that 
  \begin{equation*}
    \mu(E_{\tau})\leq \sum_{k=-1}^{\bar{k}_{\lambda/\tau}}\mu(\tilde{E}_{k}) \lesssim \Bigl( \frac{\lambda}{\tau}\Bigr)^{r}  2^{\bar{k}_{\lambda/\tau}(r(1+\epsilon)-p)}|I_{D}| \lesssim \Bigl( \frac{\lambda}{\tau}\Bigr)^{(r-1)\frac{p}{p-1-\epsilon}} |I_{D}| \lesssim  \Bigl( \frac{\lambda}{\tau}\Bigr)^{q} |I_{D}| 
  \end{equation*}
  where the last inequality holds since $\epsilon>0$ is arbitrary and $\tau\lesssim\lambda$. On the other hand
  \begin{equation*}
    \big\| \1_{(D \setminus K_\lambda)\setminus E_{\tau}} \|F\|_{X^3} \big\|_{S^\infty} \lesssim \sum_{k=-1}^{\infty} \big\| \1_{(D \setminus K_\lambda)\setminus \tilde{E}_k} \|F_k\|_{X^3} \big\|_{S^\infty} \lesssim \sum_{k=-1}^{\infty} 2^{-\varepsilon k}\tau \lesssim \tau
  \end{equation*}
  and this concludes the proof.
  \end{proof}

\begin{proof}[Proof of Theorem \ref{thm:intro-embeddings}]
  The argument is identical for the iterated and non-iterated embeddings, so we only show the iterated case. By Corollary \ref{cor:outer-Lp-domination}, using that $X$ is UMD, for any convex $\AA\subset \tPP$ it holds that
  \begin{equation*} 
    \|\1_\AA \mc{E}[f] \|_{L_\nu^{p} \sL_\mu^q \RS} \lesssim \big\|\1_\AA\, \| \mc{E}[f] \|_{X^3} \big\|_{L_\nu^{p} \sL_\mu^q S^\infty} \qquad \forall f \in \Sch(\WW;X),
  \end{equation*}
  and by the iterated embeddings for $S^\infty$ (Theorem \ref{theorem:iterated-embeddings}), using that $X$ is $r$-Hilbertian,
  \begin{equation*}
    \bigl\|\1_{\AA}\, \|\mc{E}[f] \|_{X_v^3} \bigr\|_{L_\nu^{p} \sL_\mu^q S^\infty} \lesssim \bigl\| \|\mc{E}[f] \|_{X_v^3} \bigr\|_{L_\nu^{p} \sL_\mu^q S^\infty}  \lesssim \|f\|_{L^p(\WW;X)} \qquad \forall f \in \Sch(\WW;X).
  \end{equation*}
  The first inequality above follows by the unconditionality property of sizes and thus of outer-$L^{p}$ quasi-norms.  This completes the proof.
\end{proof}


\section{Applications to the tritile form}
\label{sec:tritile-form}
Again we consider three Banach spaces $X_0$, $X_1$, $X_2$ and a bounded trilinear form $\map{\Pi}{X_0 \times X_1 \times X_2}{\CC}$.
Each $X_u$ is assumed to be UMD and $r_u$-Hilbertian for some $r_u \in [2,\infty)$.
Recall that the tritile form is the trilinear form $\Lambda_\Pi \colon \prod_{u=0}^2 \Sch(\WW;X_u) \to \CC$ defined by
\begin{equation*}
  \Lambda_\Pi(f_0,f_1,f_2)
  := \sum_{\mb{P} \in \tPP} \Pi \Big( \langle f_0; w_{\mb{P}_{0}} \rangle,  \langle f_1; w_{\mb{P}_{1}} \rangle, \langle f_2; w_{\mb{P}_{2}} \rangle \Big) |I_{\mb{P}}|.
\end{equation*}
Using the embedding theorems from the previous section, we will establish $L^p$-bounds and sparse domination for $\Lambda_\Pi$.

\subsection{$L^p$ bounds}\label{sec:Lpbounds}

\begin{proof}[Proof of Theorem \ref{thm:intro-tritile-boundedness}]
  The condition \eqref{eqn:condn} guarantees the existence of a Hölder triple $(q_0,q_1,q_2)$ such that
  \begin{equation*}
    q_u > \min(p_u,r_u)^\prime(r_u - 1) 
  \end{equation*}
  for all $u \in \{0,1,2\}$, and then by Theorem \ref{thm:intro-embeddings} we have
  \begin{equation}\label{eqn:crit1}
    \| \mc{E} [f_u] \|_{L_\nu^{p_u} \sL_\mu^{q_u} \RS} \lesssim \| f_u \|_{L^{p_u}(\WW;X_u)} \qquad \forall f_u \in \Sch(\WW;X_u)
  \end{equation}
  for all $u$.
  By Remark \ref{rmk:tritile-reduction} this suffices to prove the theorem.
\end{proof}

The set of exponents $(p_u)_{u\in\{0,1,2\}}$ to which Theorem \ref{thm:intro-tritile-boundedness} applies (more precisely, the set of reciprocals $(1/p_u)_{u\in\{0,1,2\}}$) can be characterised as the interior of a polygon.
Let $\beta_u = 1/p_u$ and $\gamma_u := 1/r_u$.
Say that $(p_0,p_1,p_2)$ is admissible if
\begin{equation*}
  \sum_{u=0}^2 \frac{1}{\min(p_u,r_u)^\prime (r_u-1)} > 1.
\end{equation*}
We rewrite the left hand side of this condition as
\begin{align*}
  \sum_{u=0}^2 \frac{1}{\min(p_u,r_u)^\prime (r_u-1)}
  &= \sum_{u=0}^{2} \frac{1}{\max(p_u^\prime,r_u^\prime) (\frac{1}{\gamma_u} - 1)} \\
  &= \sum_{u=0}^{2} \min(1-\beta_u, 1-\gamma_u)\frac{\gamma_u}{1-\gamma_u}.
\end{align*}
It follows that an admissible exponent $(p_0,p_1,p_2)$ exists only if
\begin{equation}\label{eqn:kappa-1}
  \rho := \bigg(\sum_{u=0}^{2} \gamma_u\bigg) - 1 > 0,
\end{equation}
and we assume this condition in what follows.
Consider the set of exponents
\begin{equation*}
  S := \bigg\{\beta \in (-\infty,1]^3 : \sum_{u=0}^{2} \beta_u = 1, \, \sum_{u=0}^{2} \min(1-\beta_u, 1-\gamma_u)\frac{\gamma_u}{1-\gamma_u} > 1\bigg\}.
\end{equation*}
This set is the interior of a polygon; the vertices of this polygon may be found by choosing $w \in \{0,1,2\} \sm \{u\}$  arbitrarily, setting $\beta_u = \gamma_u$, and making $\beta_w > \gamma_w$ as large as possible.
Let $v$ be the single element of $\{0,1,2\} \sm \{u,w\}$, so that $1 - \beta_v =  \beta_u + \beta_w = \gamma_u + \beta_w$.
Then the second condition in the definition of $S$, for $\beta_w > 1 - \gamma_v - \gamma_u$, becomes
\begin{equation*}
  \gamma_u + (1-\beta_w)\frac{\gamma_w}{1-\gamma_w} + \gamma_v > 1.
\end{equation*}
Rearranging this gives
\begin{equation*}
  \beta_w < \gamma_w + \rho\bigg(\frac{1}{\gamma_w} - 1\bigg),
\end{equation*}
so the vertices of $\partial S$ are given by the $6$ points $\beta$ in the Hölder triangle determined by their $(u,w)$-components
\begin{equation}\label{eqn:vertices}
  (\beta_u,\beta_w) = (\kappa_u, \gamma_w + \rho(\gamma_w^{-1} - 1)) \qquad (u \neq w \in \{0,1,2\}).
\end{equation}
The region of exponents $(\beta_u) = (p_u^{-1})$ to which Theorem \ref{thm:intro-tritile-boundedness} applies is thus the interior of the convex hull of the $6$ points in \eqref{eqn:vertices}, intersected with the cube $(0,1)^3$ (noting that $S$ generally contains some exponents with nonpositive entries).

Thus, comparing our result with that of Hytönen, Lacey, and Parissis \cite{HLP13}, we see that we obtain the same $L^p$ bounds for the tritile operator as they do for the quartile operator when restricted to the reflexive range $p_u \in (1,\infty)$.\footnote{As mentioned in the introduction, the sparse domination obtained in Theorem \ref{thm:intro-sparse-bounds} implies the range of estimates corresponding to those obtained in \cite{HLP13}. Once more we thank the anonymous referee for this observation.}

\subsection{Sparse domination}

\begin{proof}[Proof of Theorem \ref{thm:intro-sparse-bounds}]
  We follow the argument in \cite[\textsection 1.4.3]{gU-thesis}.

  We will show the following abstract sparse domination result: for any Hölder triple $(q_{u})_{u\in\{0,1,2\}}$ and any triple of exponents $p_{u}\in[1,\infty)$, we have the bound 
  \begin{equation}\label{eqn:abstract-domination}
    \begin{aligned}
      \sum_{\mb{P}\in \tPP} \Bigl| &\Pi^* \bigl(F_{0}(\mb{P}),F_{1}(\mb{P}),F_{2}(\mb{P})\bigr)\Bigr|\,|I_{\mb{P}}| \\
      &\lesssim \sup_{\| \mc{G} \|_{sp}\leq1} \sum_{I\in\mc{G}}|I| \prod_{u=0}^{2}|I|^{-1/p_{u}}\|\1_{D(I)} F_{u} \|_{L^{p_{u}}_{\nu}\sL^{q_{u}}_{\mu}\RS}.
    \end{aligned}
    \end{equation}
    for any $F_{u}\in\Bor(\tPP;X_{u}^3)$.
    This result suffices to prove the theorem; to see this, let $F_{u}=\Emb[f_{u}]$ and notice that $F_u = \Emb[f_u \1_I]$ on $D(I)$.
    Thus by Theorem \ref{thm:intro-embeddings}, choosing the H\"older triple $(q_u)_{u \in \{0,1,2\}}$ such that $q_u > \min(p_u,r_u)'(r_u - 1)$ for each $u$ (such a choice is possible by condition \eqref{eqn:condn}), the bound 
    \begin{equation*}
      |I|^{-1/p_{u}}\|\1_{D(I)} F_{u} \|_{L^{p_{u}}_{\nu}\sL^{q_{u}}_{\mu}\RS} \lesssim \| f_{u} \|_{\sL^{p_{u}}(I;X)}
    \end{equation*}
    holds.
    Since $|\Lambda_{\Pi}(f_{0},f_{1},f_{2})|\leq \sum_{\mb{P}\in \tPP} \Bigl| \Pi^*\bigl(F_{0}(\mb{P}),F_{1}(\mb{P}),F_{2}(\mb{P})\bigr)\Bigr|\,|I_{\mb{P}}| $ this implies the conclusion of the theorem.

    It remains to show that \eqref{eqn:abstract-domination} holds. The definition of the iterated outer-$L^p$ quasinorms implies that for every strip  $D$, there exists a subset $K_D \subset D$ such that
  \begin{align}
    \label{eqn:f-iter-bd}
    \|  \1_{D \setminus K_D}\,F_{u} \|_{\sL_\mu^{q_u} \RS} &\lesssim \nu(D)^{-1/p_{u}}\|\1_{D}F_{u}\|_{L^{p_{u}}_{\mu}\sL^{q_{u}}\RS} \quad \forall u \in \{0,1,2\},
    \\
    \label{eqn:sparse-meas}
         \nu(K_D) &\leq \nu(D)/2.
  \end{align}
  Without loss of generality we may assume that $K_{D}$ is a union of strips, i.e. $K_{D}=\bigcup_{I\in\mc{J}(I_{D})}D(I)$, and that these strips are pairwise disjoint. 
  Set $\mc{G}_0 = \{I_{0}\}$ for some initial interval.  Proceed iteratively: having defined a collection of intervals $\mc{G}_n$, define
  \begin{equation*}
   \mc{G}_{n+1} := \bigcup_{I\in\mc{G}_{n}} \mc{J}(I )
  \end{equation*}
  where $\mc{J}(I)$ is the set of intervals defined by \eqref{eqn:f-iter-bd} and \eqref{eqn:sparse-meas} (with $D = D(I)$). The bound \eqref{eqn:sparse-meas} guarantees, by induction, that
  \begin{equation*}
    \max_{I\in\mc{G}_{n}} |I|\leq 2^{-n}|I_{0}|
  \end{equation*}
  and thus $\bigcap_{n\in\N}\bigcup_{I\in\mc{G}_{n}} I=\emptyset$.   Let $(q_{u})_{u\in\{0,1,2\}}$ be any Hölder triple; using the Hölder inequality for $L_\mu^{q_{u}} \RS$ gives us
  \begin{equation}\label{eqn:strip-sp-dom}
    \begin{aligned}
      &\sum_{\mb{P}\in D(I_{0})} \Bigl| \Pi^*\bigl(F_{0}(\mb{P}),F_{1}(\mb{P}),F_{2}(\mb{P})\bigr)\,|I_{\mb{P}}| \Bigr|
      \\
      &=
      \sum_{n=0}^{\infty}\sum_{I\in\mc{G}_{n}} \sum_{\mb{P}\in\tPP} \Bigl| \1_{D(I)\setminus K_{D(I)}}\Pi^*\bigl(F_{0}(\mb{P}),F_{1}(\mb{P}),F_{2}(\mb{P})\bigr)\Bigr|\,|I_{\mb{P}}|  
      \\
      &
      \lesssim \sum_{n=0}^{\infty}\sum_{I\in\mc{G}_{n}} \prod_{u=0}^{2}\|\1_{D(I)\setminus K_{D(I)}} F_{u} \|_{L^{q_{u}}_{\mu}\RS}
      \\
      &\lesssim \sum_{n=0}^{\infty}\sum_{I\in\mc{G}_{n}}|I| \prod_{u=0}^{2}|I|^{-1/p_{u}}\|\1_{D(I)} F_{u} \|_{L^{p_{u}}_{\nu}\sL^{q_{u}}_{\mu}\RS}.
    \end{aligned}
  \end{equation}
  Recall that for any $I\in\mc{G}$ we have set  $K_{D(I)}=\bigcup_{I'\in\mc{J}(I)} D(I') $ so that \eqref{eqn:f-iter-bd} holds and guarantees the last bound. 

  We now show that $\mc{G}=\bigcup_{n\in\N}\mc{G}_{n}$ is sparse with $\| \mc{G} \|_{sp}\leq 1$.     The intervals of $\mc{G}$ are nested in the sense that if $J\in\mc{G}_{n+1}$ then there exists $J'\supset J$ with $J'\in\mc{G}_{n}$.  First suppose $I\in\mc{G}_{n_{0}}$ for some $n_{0}\in \NN$; it follows by induction from \eqref{eqn:sparse-meas} that
  \begin{equation*}
    \sum_{ \substack{J\subset I\\J\in\mc{G}_{n_{0}+k} }} |J| \leq |I|/2^{k}.
  \end{equation*}
  and thus
  \begin{equation*}
    \sum_{ \substack{J\subset I\\J\in\mc{G}}} |J|=\sum_{k=1}^{\infty} \sum_{ \substack{J\subset I\\J\in\mc{G}_{n_{0}+k} }} |J| \leq \sum_{k=0}^{\infty}|I|/2^{k}= |I|.
  \end{equation*}
  If $I\notin \mc{G}$, then there exists $n_{0}\in\NN$ and disjoint intervals $I_{m}\subset I$, $I\in\mc{G}_{n_{0}}$ such that
  \begin{equation*}
  \{J\subset I \colon J\in\mc{G} \}=\bigcup_{m\in\NN} \{J\subset I_{m}\colon J\in\mc{G}\}.
  \end{equation*}
  Thus $\| \mc{G} \|_{sp}\leq 1$.

  For any $F_{u}\in\Bor(\tPP;X_{u}^3)$ with $u\in\{0,1,2\}$ we can write
  \begin{equation*}
     \sum_{\mb{P} \in \tPP} \Bigl | \Pi^*\bigl(F_{0}(\mb{P}),F_{1}(\mb{P}),F_{2}(\mb{P})\bigr)\Bigr||I_{\mb{P}}| 
     \leq \sup_{D_{0}\in\DD}\sum_{\mb{P}\in D_{0}} \Bigl| \Pi^*\bigl(F_{0}(\mb{P}),F_{1}(\mb{P}),F_{2}(\mb{P})\bigr)\Bigr||I_{\mb{P}}| .
   \end{equation*}
   Estimating the sum over $D_0$ via \eqref{eqn:strip-sp-dom} and using that $\|\mc{G}\|_{sp} \leq 1$ for the collections $\mc{G}$ that we constructed shows that \eqref{eqn:abstract-domination} holds, and completes the proof.
\end{proof}



\section{Appendix: Arguing via $R$-bounds and RMF}
\label{sec:RMF}
In this appendix we present an alternative approach to our main results.
In this approach, the randomised sizes are simplified; the defect operator does not appear, and the sizes resemble more closely the sizes used in \cite{HLP13} and \cite{DPO18} or in earlier scalar-valued proofs.
For this simplicity we pay the price of having to manipulate $R$-bounds, which leads to the assumption of the RMF property on the trilinear form $\Pi$.
The same technical difficulty occured in \cite{DPO18}; a new method of obtaining these results without RMF was recently given in \cite{DPLMV19}.

\subsection{$R$-bounds and the RMF property}\label{sec:RMF}

First we recall the definition of \emph{$K$-convexity} of a Banach space.
All the spaces that we consider, being UMD spaces, are $K$-convex \cite[Proposition 4.3.10]{HNVW16}.

\begin{defn}
  A Banach space $X$ is \emph{$K$-convex} if for all finite sequences $(x_n)_{n=1}^N$ in $X$,
\begin{equation*}
  \E \bigg\|\sum_{n=1}^N \varepsilon_n x_n \bigg\|_X
  \simeq \sup_{(x_n^*)} \bigg| \sum_{n=1}^N x_n^*(x_n) \bigg|
\end{equation*}
where the supremum is taken over all sequences $(x_n^*)_{n=1}^N$ in $X^*$ such that
\begin{equation*}
  \E \bigg\|\sum_{n=1}^N \varepsilon_n x_n^* \bigg\|_{X^*} = 1,
\end{equation*}
with implicit constant independent of $N$.
\end{defn}

As a technical tool, we use a strong notion of boundedness of a set of operators known as \emph{$R$-boundedness}.
For a short history of this concept see \cite[\textsection 8.7]{HNVW17}.

\begin{defn}\label{dfn:R-bound}
Let $X$ and $Y$ be Banach spaces, and let $\mc{T} \subset \mc{L}(X,Y)$ be a set of bounded operators from $X$ to $Y$. We say that the set $\mc{T}$ is \emph{$R$-bounded} if there exists a constant $C > 0$ such that for all finite sequences $(x_n)_{n=1}^N$ in $X$ and $(T_n)_{n=1}^N$ in $\mc{T}$, the estimate
\begin{equation}\label{eqn:R-bound}
  \E \bigg\| \sum_{n=1}^N \varepsilon_n T_n x_n \bigg\|_Y \leq C \E \bigg\| \sum_{n=1}^N \varepsilon_n x_n \bigg\|_X
\end{equation}
holds. 
The smallest allowable $C$ in this estimate is called the \emph{$R$-bound} of $\mc{T}$, and denoted by $R(\mc{T})$. 
\end{defn}

If $\mc{T}$ is $R$-bounded, then $\mc{T}$ is uniformly bounded in norm (consider $N=1$ in \eqref{eqn:R-bound}).
If $Y$ is $K$-convex, then $\mc{T}$ is $R$-bounded if and only if
\begin{equation*}
  \bigg| \sum_{n=1}^N y_n^*(T_n x_n) \bigg| \leq C\, \E \bigg\| \sum_{n=1}^N \varepsilon_n x_n \bigg\|_X
  \E \bigg\| \sum_{n=1}^N \varepsilon_n y_n^* \bigg\|_{Y^*} 
\end{equation*}
for all sequences $(x_n)$ in $X$, $(y_n^*)$ in $Y^*$, and $(T_n)$ in $\mc{T}$, and the smallest admissible constant $C$ here is equivalent to $R(T)$.
This is analogous to the Hölder inequality
\begin{equation*}
  \bigg| \sum_{n \in \NN} f(n)g(n)h(n) \bigg| \leq \|f\|_{\ell^\infty} \|g\|_{\ell^2} \|h\|_{\ell^2}
\end{equation*}
for sequences $f,g,h \colon \NN \to \CC$, keeping in mind the analogy between Rademacher sums and square functions.

The concept of $R$-boundedness applies to sets of operators, but we can apply it to sets of vectors by viewing them as operators.
This relies on the additional information of an identification of vectors with operators.

\begin{defn}\label{defn:Rbd-embedding}
  Let $X$ be a Banach space, and consider an embedding (i.e. a bounded linear injective map) $\iota \colon X \hookrightarrow \mc{L}(Y,Z)$, where $Y$ and $Z$ are Banach spaces and $\mc{L}(Y,Z)$ is the Banach space of bounded linear operators from $Y$ to $Z$.
  Then for all subsets $V \subset X$ we define the $R$-bound of $V$ with respect to the embedding $\iota$ by
  \begin{equation*}
    R_\iota(V) := R(\iota(V)).
  \end{equation*}
\end{defn}

This definition leads to an analogue of the Doob (or dyadic Hardy--Littlewood) maximal function, with $R$-bounds replacing norm bounds.

\begin{defn}
  Let $X$ be a Banach space, and consider an embedding $\iota \colon X \to \mc{L}(Y,Z)$ for some Banach spaces $Y$ and $Z$.  Let $(\mc{F}_n)_{n \in \N}$ be a filtration on a $\sigma$-finite measure space $(\Xi,\mc{F},\mu)$.
  The \emph{Rademacher maximal operator} $\mc{M}_\iota$ with respect to $\iota$ is defined on $\mc{F}$-measurable functions $\map{f}{\Xi}{X}$ by
  \begin{equation*}
    \mc{M}_{\iota}f(x) := R_\iota(\{\E[f|\mc{F}_n](x) : n \in \NN\}) \qquad  \forall x \in \Xi,
  \end{equation*}
  where $\E[f|\mc{F}_n]$ is the conditional expectation of $f$ on $\mc{F}_n$.
\end{defn}

The $L^p$-boundedness of this maximal function is a geometric property of the embedding $\iota$ (and thus of the Banach spaces $X,Y,Z$), which may or may not hold.
Thus it is given a name.

\begin{defn}
  Let $X$ be a Banach space, and $\iota \colon X \to \mc{L}(Y,Z)$ an embedding of $X$ into the bounded linear operators between some Banach spaces $Y$ and $Z$.
  We say that $\iota$ has the \emph{RMF property} if for all filtrations and measure spaces as above, $\mc{M}_{\iota}$ is bounded from $L^p(\Xi;X)$ to $L^p(\Xi)$ for all $p \in (1,\infty)$.
  More concisely, we say that $\iota$ is an \emph{RMF embedding of $X$}, without making explicit reference to the spaces $Y$ and $Z$.
\end{defn}

The RMF property is independent of the choice of filtraton $\mc{F}_n$ and $\sigma$-finite measure space $(\Xi,\mc{F},\mu)$ \cite[Theorem 5.1]{mK11}.
Furthermore, if $\iota$ is RMF, then $Z$ is $K$-convex by \cite[Proposition 4.2]{mK11} (using the equivalence of $K$-convexity and non-trivial type \cite[Theorem 7.4.23]{HNVW17}).

We say that a Banach space $X$ has the RMF property (without reference to an embedding) if the natural embedding $\iota \colon X \to \mc{L}(X^*,\CC)$ has the RMF property.
In this case, for a function $\map{f}{\RR}{X}$ on the line equipped with the dyadic filtration, we have 
\begin{equation*}
  \mc{M}_{\iota}f(x) = \sup \bigg\{\E \bigg\| \sum_{I \ni x} \varepsilon_I \lambda_I \langle f \rangle_I \bigg\|_X \bigg\} \qquad \forall x \in \RR
\end{equation*}
where the sum is over all dyadic intervals $I$ containing $x$, and where the supremum is taken over all normalised sequences $\lambda \in \ell^2(\mc{D})$ on the set $\mc{D}$ of dyadic intervals in $\RR$.
This form of the RMF property was first introduced by Hyt\"onen, McIntosh, and Portal \cite{HMcP08},\footnote{The form we use here, in which the Banach space $X$ is identified with a space of operators via an embedding, seems to have been introduced by Kemppainen \cite{mK11}.} who proved that the following classes of Banach spaces are RMF:
\begin{itemize}
\item all spaces of type $2$,
\item all UMD Banach lattices (including $L^p$-spaces with $p \in (1,\infty)$),
\item all noncommutative $L^p$ spaces with $p \in (1,\infty)$, and in particular the Schatten classes $\mc{C}^p$ with $p \in (1,\infty)$. 
\end{itemize}
They also proved that $\ell^1$ does not have the RMF property.
It is not known whether the UMD property implies the RMF property.
The converse fails: there exists a space of type $2$ (and therefore with the RMF property) which is not reflexive, and therefore not UMD \cite{rJ78}.

Now consider a triple of Banach spaces $X_0$, $X_1$, $X_2$ and a bounded trilinear form $\map{\Pi}{X_0 \times X_1 \times X_2}{\CC}$.
Indexing $\{0,1,2\} = \{u,v,w\}$ arbitrarily, the trilinear form induces natural embeddings $\iota_\Pi^u \colon X_u \to \mc{L}(X_v,X_w^*)$.
Thus for a set of vectors $V \subset X_u$ we have an $R$-bound $R_\Pi(V) := R_{\iota_\Pi^u}(V)$ and a Rademacher maximal operator $\mc{M}_\Pi := \mc{M}_{\iota_\Pi^u}$.
We generally omit $\iota$ and $u$ in our notation, and observe that the indexing $\{0,1,2\} \sm \{u\}$ does not affect these quantities, as $R$-bounds are preserved under taking adjoints.
We say that $\Pi$ has the RMF property if each embedding $\iota_\Pi^u$ is RMF.

If each $X_u$ is $K$-convex (as is the case when each $X_u$ is UMD) then for all subsets $V \subset X_u$, the $R$-bound $R_{\Pi}(V)$ is equivalent to the smallest constant $C$ such that
\begin{equation}\label{eqn:Rbd-tri}
  \bigg| \sum_{n=1}^N \Pi(x_{0,n}, x_{1,n}, x_{2,n}) \bigg| \leq C \prod_{v \neq u} \E \bigg\| \sum_{n=1}^N \varepsilon_n x_{v,n} \bigg\|_{X_v}
\end{equation}
for all finite sequences $(x_{u,n})_{n=1}^N \subset V$, and $(x_{v,n})_{n=1}^N \subset X_v$, $v \in \{0,1,2\} \sm \{u\}$.

\begin{rmk}
	Although all known UMD spaces have the RMF property, this does not tell us that a trilinear form $\Pi \colon X_0 \times X_1 \times X_2 \to \CC$ automatically has the RMF property when each $X_u$ is UMD.
	If each $X_u$ is a UMD Banach function space and $\Pi$ is the pointwise product, then by the Khintchine--Maurey theorem \eqref{eqn:khintchine-maurey} one can reduce matters to boundedness of the lattice maximal function, which follows from UMD, and so $\Pi$ has the RMF property.
	However, a natural trilinear form is given by composing the composition map on Schatten classes $\mc{C}^p \times \mc{C}^q \times \mc{C}^r \to \mc{C}^1$ (when $p^{-1} + q^{-1} + r^{-1} = 1$ with the trace map $\mc{C}^1 \to \CC$; here each of the Banach spaces is UMD (and each has RMF, as an individual Banach space) but it is not known whether the constructed trilinear form has the RMF property.
\end{rmk}

\subsection{Randomised sizes with $R$-bounds}

We can use the notion of $R$-boundedness in the previous section to define new randomised sizes.
The key difference between these sizes and those defined in Section \ref{sec:outermeasures} are that the defect operator is not used, and that $R$-bounds appear in the overlapping sizes. 

In this section we always assume that  $X_0$, $X_1$, and $X_2$ are $K$-convex Banach spaces, and $\Pi \colon X_0 \times X_1 \times X_2 \to \CC$ is a bounded trilinear form.

\begin{defn}[Randomised sizes]\label{def:randomized-size-R}
  For $u\in\{0,1,2\}$ the $X_u$-size $\RS_u$ is given by
  \begin{equation*}
    \|F\|_{\RS_u(T)} := \sum_{v\in\{0,1,2\}} \|F\|_{\RS_{u,v}(T)} \qquad \forall F \in \Bor(\tPP;X_u),
  \end{equation*}
  where 
  \begin{equation}\label{eq:randomized-size}
    \|F\|_{\RS_{u,v}(T)} := \begin{cases}
      \bigg( \fint_{I_T} R_\Pi(\{F(\mb{P})\1_{I_{\mb{P}}}(x) : \mb{P} \in T^u\})^3 \, \dd x \bigg)^{1/3} & (v = u) \\
      \bigg( \fint_{I_T} \E \bigg\| \sum_{\mb{P} \in T^v} \varepsilon_{\mb{P}} F(\mb{P}) \1_{I_{\mb{P}}}(x) \bigg\|_{X}^3 \, \dd x \bigg)^{1/3}& (v \neq u).
    \end{cases}
  \end{equation}
\end{defn}

	Given $F_u \in \Bor(\tPP;X_u)$, define the function $\Pi(F_0,F_1,F_2) \colon \tPP \to \CC$ by
	\begin{equation*}
		\Pi(F_0,F_1,F_2)(\mb{P}) := \Pi(F_0(\mb{P}),F_1(\mb{P}), F_2(\mb{P})).
	\end{equation*}

\begin{prop}\label{prop:size-holder-R}
	For all $F_u \in \Bor(\tPP;X_u)$,
  \begin{equation*}
    \| \Pi(F_0,F_1, F_2) \|_{S^1(T)} \leq \prod_{u=0}^{2} \|F_u\|_{\RS_{u}(T)} \qquad \forall T \in \TT.
  \end{equation*}
 A corresponding H\"older inequality for outer-Lebesgue spaces follows as in Section \ref{sec:outermeasures}.
\end{prop}

\begin{proof}
  First note that
  \begin{equation*}
    \sum_{\mb{P} \in T} |\Pi(F_0(\mb{P}), F_1(\mb{P}), F_2(\mb{P}))||I_{\mb{P}}|
    \leq \sum_{u=0}^2 \sum_{\mb{P} \in T^u} |\Pi(F_0(\mb{P}), F_1(\mb{P}), F_2(\mb{P}))||I_{\mb{P}}|,
  \end{equation*}
  so it suffices to fix $u \in \{0,1,2\}$ and deal with the summands individually.
  Fix a normalised sequence $a \in \ell^\infty(T^u;\CC)$ and estimate by duality
   \begin{align*}
     &\bigg| \sum_{\mb{P} \in T^u} a(\mb{P}) \Pi(F_0(\mb{P}), F_1(\mb{P}), F_2(\mb{P})) |I_{\mb{P}}| \bigg| \\
     &\leq \int_{I_T} \bigg| \sum_{\mb{P} \in T^u} a(\mb{P})  \Pi( F_0(\mb{P}) \1_{I_\mb{P}}(x), F_1(\mb{P}) \1_{I_\mb{P}}(x), F_2(\mb{P}) \1_{I_\mb{P}}(x)) \bigg| \, \dd x \\ 
     &\lesssim \int_{I_T} R_{\Pi}(\{F_u(\mb{P})\1_{I_{\mb{P}}}(x) : \mb{P} \in T^u \}) \prod_{v\neq u} \E \bigg\| \sum_{\mb{P} \in T^u} \varepsilon_{\mb{P}} F_v(\mb{P}) \1_{I_\mb{P}}(x) \bigg\|_{X_v} \, \dd x \\
    &\leq |I_T| \prod_{v=0}^2 \|F_v\|_{\RS_{v}(T)}
   \end{align*}
   using \eqref{eqn:Rbd-tri}. 
 \end{proof}
 
 The key ingredient that we need is a size domination result for these sizes.
 This uses the deterministic size $S^\infty$ from Section \ref{sec:outermeasures}. 
 
 \begin{thm}\label{thm:size-domination-R}
   For all $u \in \{0,1,2\}$ and all $f \in \Sch(\WW;X_u)$, and all convex $\AA\subset \tPP$,
  \begin{equation*}
    \bigl\|\1_{\AA} \mc{E}_u[f]\bigr\|_{\RS_u} \lesssim \big\|\1_{\AA} \| \mc{E}[f] \|_{X_u^{3}}\big\|_{S^\infty},
  \end{equation*}
  with implicit constant independent of $\AA$.
\end{thm}

Corresponding outer-Lebesgue estimates as in Corollary \ref{cor:outer-Lp-domination} follow immediately.


\begin{proof}[Proof of Theorem \ref{thm:size-domination-R}]
  As in the proof of Proposition \ref{prop:lacunary-size-domination}, it suffices to replace $\tPP$ with $\tPP_{N}=\{P \in \tPP : 3^{-N}<|I_{P}|<3^{N}\}$.
  Fix a tree $T$.
  Since $\AA\cap T \cap\tPP_{N} $ is finite and convex, by Lemma \ref{lem:convex-proj} there exists a function $g \in \Sch(\WW;X_u)$ supported on $I_T$ such  that
  \begin{align*}
      &\|g\|_{L^{\infty}(I_{T};X_u)} \lesssim \sup_{\substack{\mb{P} \in  T \cap\AA \cap \tPP_{N} \\ v \in \{0,1,2\}}}
      \|\langle f ; w_{\mb{P}_v}\rangle\|_{X_u}\\
      & \langle g; w_{\mb{P}_{v}}\rangle = \langle f; w_{\mb{P}_{v}}\rangle \qquad \forall P\in  T \cap \AA \cap \tPP_{N}, \;v\in\{0,1,2\}.
  \end{align*}
It suffices to bound each of the summands in
  \begin{equation*}
    \|\1_{\AA \cap \tPP_{N}} \mc{E}_u[f] \|_{\RS_u(T)} = \|\1_{\AA \cap \tPP_{N} } \mc{E}_u[f]\|_{\RS_{u,u}(T)} + \sum_{v \neq u} \|\1_{\AA \cap \tPP_{N}} \mc{E}_u[f]\|_{\RS_{u,v}(T)}
  \end{equation*}
  by $\|g\|_{L^{\infty}(I_{T};X_u)}$.
  The lacunary parts $\RS_{u,v}$ ($v \neq u$) are already handled in the proof of Proposition \ref{prop:lacunary-size-domination}, using the UMD property (one need only replace the exponent $2$ by the exponent $3$).
  It remains to treat the overlapping part, $\RS_{u,u}$.
  Since $\xi_{T}\in w_{\mb{P}_{u}}$, we have
  \begin{equation*}
  \exp_{(\xi_{\mb{P}_{u}}-\xi_{T})}(x-x_{\mb{P}})=1 \qquad \forall x\in I_{\mb{P}}  
  \end{equation*}
  and so
  \begin{equation*}
    w_{\mb{P}_{u}} = \exp_{(\xi_{\mb{P}_{u}}-\xi_{T})}(x_{\mb{P}}) \exp_{\xi_{T}}(x)|I_{\mb{P}}|^{-1}\1_{I_{\mb{P}}}.
  \end{equation*}
  Writing out the size and using the RMF property of $\Pi$ gives
  \begin{align*}
    &\|\1_{\AA \cap \tPP_{N}} \mc{E}_u[f] \|_{\RS_{u,u}(T)}^3 \\
    &\leq \fint_{I_T} R_\Pi \Big(\big\{\langle g; w_{\mb{P}_u} \rangle \1_{I_{\mb{P}}}(x) : \mb{P} \in T^u\cap\AA \cap \tPP_{N}\big\}\Big)^3 \, \dd x \\
    &= \fint_{I_T} R_\Pi \Big(\big\{\exp_{(\xi_{\mb{P}}-\xi_{T})}(x_{\mb{P}}) \langle \exp_{-\xi_T} g\rangle_{I_{\mb{P}}}   : \mb{P} \in T^u\cap \AA \cap\tPP_{N}, x \in I_{\mb{P}}\big\}\Big)^3 \, \dd x \\
    &\simeq \fint_{I_T} \mc{M}_{\Pi} (\exp_{-\xi_T} g)(x)^3 \, \dd x \\
    &\lesssim |I_T|^{-1} \|\exp_{-\xi_T} g\|_{L^3(I_T;X)}^3 \leq \|g\|_{L^{\infty}(I_{T};X)}
  \end{align*}
  where $\mc{M}_\Pi$ is the Rademacher maximal operator with respect to $\Pi$, and we used the contraction principle  to remove the unimodular coefficients $\exp_{(\xi_{\mb{P}}-\xi_{T})}(x_{\mb{P}})$.
\end{proof}

Having proven this size domination result, we are reduced to the situation of Section \ref{sec:energy-embedding}, where embedding bounds are proven with respect to the deterministic size $S^\infty$, which is the same in this formulation.
Thus we obtain an alternative version of our main theorem, with different sizes and an additional RMF assumption.

\begin{thm}\label{thm:intro-embeddings-R}
  Let $X_0$, $X_1$, and $X_2$ be UMD Banach spaces, such that each $X_u$ is $r_u$-Hilbertian for some $r_u \in [2,\infty)$.
  Let $\Pi \colon X_0 \times X_1 \times X_2 \to \CC$ be a bounded trilinear form with the RMF property.
  Then for all $u \in \{0,1,2\}$ and all convex sets $\AA \subset \tPP$ of tritiles, the following embedding bounds hold.
  \begin{itemize}
  \item For all $p \in (r_u,\infty)$,
    \begin{equation}\label{eqn:intro-emb-1-R} 
      \bigl\|\1_{\AA}\,\Emb_u[f] \bigr\|_{L_\mu^p \RS_u} \lesssim \|f \|_{L^p(\WW;X_u)} \qquad \forall f \in \Sch(\WW;X_u).
    \end{equation}
  \item For all $p \in (1,\infty)$ and all $q \in (\min(p,r_u)^\prime(r-1),\infty)$,
    \begin{equation}\label{eqn:intro-emb-2-R}
      \bigl\|\1_{\AA}\,\Emb_u[f]\bigr\|_{L_\nu^{p} \sL_\mu^q \RS_u} \lesssim \|f\|_{L^p(\WW;X_u)} \qquad \forall f \in \Sch(\WW;X).
    \end{equation}
  \end{itemize}
  The implicit constants in the above bounds do not depend on $\AA$.
\end{thm}

As the sizes $\RS_u$ satisfy a H\"older inequality (Proposition \ref{prop:size-holder-R}), we obtain an alternative proof of our results on the tritile operator (Theorems \ref{thm:intro-tritile-boundedness} and \ref{thm:intro-sparse-bounds}) under the additional assumption that $\Pi$ has the RMF property.



\footnotesize
\bibliographystyle{amsplain}
\bibliography{main} 
\end{document}